\documentclass[11pt,reqno]{amsart}
\usepackage{amsthm, amsmath, amssymb, amscd, latexsym, multicol, verbatim, enumerate, caption, subcaption, color, adjustbox}
\usepackage{tikz}
\usetikzlibrary{arrows}
\usepackage{graphicx}
\usepackage{float}
\usepackage[toc,page]{appendix}
\usepackage[utf8]{inputenc}
\usepackage{hyperref}
\hypersetup{colorlinks,linkcolor={blue},citecolor={blue},urlcolor={blue}}
\usepackage[backend=bibtex,style=numeric,doi=false,isbn=false,url=false,giveninits=true,maxbibnames=50]{biblatex}
\advance\textwidth by 1.2in \advance\oddsidemargin by -.6in \advance\evensidemargin by -.6in \parskip=.1cm

\theoremstyle{plain}
\newtheorem{thm}{Theorem}[section]
\theoremstyle{definition} 
\newtheorem{lem}[thm]{Lemma}
\newtheorem{cor}[thm]{Corollary}

\newtheorem{prop}[thm]{Proposition}
\newtheorem{exa}[thm]{Example}
\newtheorem{defn}[thm]{Definition}
\newtheorem{rem}[thm]{Remark}
\newtheorem*{thmx}{Theorem}
\newtheorem*{exax}{Example}
\newtheorem*{corx}{Corollary}

\numberwithin{equation}{section}

\usepackage{lineno}

\newcounter{cnt}
 \makeatletter
\def\mydggeometry{\makeatletter\dg@YGRID=1\dg@XGRID=20\unitlength=0.003pt\makeatother}
\makeatother \theoremstyle{remark}
\numberwithin{equation}{section}
\newcommand{\nc}{\newcommand}
\nc{\lie}[1]{\mathfrak{#1}} \nc\bp{\mathbf{p}} \nc\bq{\mathbf{q}} \nc\mD{\mathbb{D}} \nc\bs{\mathbf{s}} \nc\bt{\mathbf{t}} \nc\bz{\mathbb{Z}} \nc\bc{\mathbb C} 
\nc\Max{\operatorname{Max}}

\newcommand{\g}{\mathfrak{g}}
\newcommand{\n}{\mathfrak{n}}

\newcommand{\hra}{\hookrightarrow}

\newcommand{\mc}{\mathbb{C}}

\newcommand{\bb}{\mathfrak{b}}
\newcommand{\h}{\mathfrak{h}}

\newcommand{\m}{\omega}
\newcommand{\s}{\mathfrak{s}}

\begin{document}

\author{Kunda Kambaso}
\title{Homogeneous bases for Demazure modules}
\address{Lehrstuhl f\"ur Algebra und Darstellungstheorie, RWTH Aachen University, 52062 Aachen, Germany}
\email{kambaso@art.rwth-aachen.de}

\begin{abstract}
We study the PBW filtration on the Demazure modules $V_{r_{\gamma}}(\lambda)$ associated to reflections $r_{\gamma}$ at positive roots in type $A_n$, long roots in type $C_n$, short roots in type $B_n$ and positive roots not involving the simple root $\alpha_{n-1}$ in type $D_n$. In each type, we construct a normal polytope labelling a basis for the associated graded space $\operatorname{gr} V_{r_{\gamma}}(\lambda)$ and prove that the annihilating ideal of $\operatorname{gr} V_{r_{\gamma}}(\lambda)$ is a monomial ideal.
\end{abstract}

\keywords{PBW filtration, Demazure modules.}
\subjclass[2010]{17B10, 16S30, 05E10 }
\maketitle

\section{Introduction}
Let $\g$ be a finite-dimensional complex Lie algebra. The PBW filtration on the universal enveloping algebra $U (\g)$ of $\g$ is defined as
\[
U_s(\g) = \operatorname{span} \{ x_{i_1} \cdots x_{i_l} : x_{i_j} \in \g , l \leq s \}.
\]
One obtains an induced filtration on any cyclic $\g \operatorname{-module} M$. The associated graded space denoted $\operatorname{gr} M$ is a module for the abelianised Lie algebra $\g^a$ and the symmetric algebra $S(\g)$. We consider the case when $\g$ is a complex simple Lie algebra with triangular decomposition $ \g = \n^+ \oplus \h \oplus \n^-$ where $\bb = \n^+ \oplus \h$ is the Borel subalgebra. Let $\lambda$ be a dominant integral weight for $\g$ and denote by $V(\lambda) = U(\n^-) v_{\lambda}$, an irreducible finite-dimensional highest weight module of $\g$. Let $\operatorname{gr} V(\lambda)$ denote the associated graded space of the PBW filtration on $V(\lambda)$ induced by the filtration on $U(\n^-)$. Then $\operatorname{gr} V(\lambda)$ is a cyclic module for the deformed Lie algebra $\bb \oplus \n^{-,a}$. The associated graded spaces $\operatorname{gr} V(\lambda)$ have been studied quite a lot in recent years in a series of papers \cite{FFLa1, FFLa2, Gor, FFLb, FFLc, BD} and many more. In type $A$ and $C$, the annihilating ideal for $\operatorname{gr} V(\lambda)$ has been computed as a $U(\n^+)$-module in \cite{FFLa1} and \cite{FFLa2} respectively and a normal polytope $P(\lambda)$ (also known as the Feigin-Fourier-Littlemann-Vinberg polytope or just FFLV polytope) labelling a basis of $\operatorname{gr} V(\lambda)$ has been constructed in each type. The inequalities defining these polytopes can be described using Dyck paths in the poset $(R^+, \leq)$ where $R^+$ is the set of positive roots of $\g$. The associated PBW degenerate projective varieties have been studied for example in \cite{Fea1, Fea2, IL} for type $A$ and in \cite{FeFiL} for type $C$. We would like to point out that constructions of bases (or polytopes) similar to the ones mentioned above can be found in \cite{BK} and \cite{M} for type $B$ and adding \cite{Gor2} for type $D$.

In the present paper, we study the PBW filtration on Demazure modules, i.e. submodules of $V(\lambda)$ generated by an extremal weight vector using the action of the Borel subalgebra $\bb = \n^+ \oplus \h$  for $\g$ of Lie type $A, B, C$ and $D$. Some steps towards the study of PBW filtrations on Demazure modules in type $A$ have been carried out in \cite{Fou1} for triangular subsets of $R^+$ and in \cite{BF} for the special case when the highest weight is a multiple of a fundamental weight. However, nothing is known about PBW filtrations of Demazure modules in other types.

We will consider Demazure modules associated to Weyl group elements $r_{\gamma}$, where $r_{\gamma}$ is a reflection at any positive root in type $A_n$, the long roots in type $C_n$, the short roots in type $B_n$ and positive roots not involving the simple root $\alpha_{n-1}$ in type $D_n$. Let $s_{\alpha_i}$ denote the reflection at the simple root $\alpha_i$. Then the elements $r_{\gamma}$ have the presentation
\[
r_{\gamma} = s_{\alpha_i} s_{\alpha_{i+1}} \cdots s_{\alpha_{n-1}} s_{\alpha_n} s_{\alpha_{n-1}} \cdots s_{\alpha_{i+1}} s_{\alpha_i}.
\]
We note that in type $A$, one has
\[
r_{\beta} = s_{\alpha_i} s_{\alpha_{i+1}} \cdots s_{\alpha_{j-1}} s_{\alpha_j} s_{\alpha_{j-1}} \cdots s_{\alpha_{i+1}} s_{\alpha_i}
\]
for any positive root $\alpha_{i,j}$. Let $R_{r_{\gamma}}^- = R^+ \cap r_{\gamma}^{-1} (R^-) \subseteq R^+$ denote the set of inversions of $r_{\gamma}$ and let $\n_{r_{\gamma}}^-$ denote the subalgebra of $\n^-$ generated by the root spaces of $- R_{r_{\gamma}}^-$. By conjugation with ${r_{\gamma}}^{-1}$, we identify $V_{r_{\gamma}}(\lambda)$ with $U(\n_{r_{\gamma}}^-) v_{\lambda}$ and study the associated graded module $\operatorname{gr} U(\n_{r_{\gamma}}^-) v_{\lambda}$ with respect to the PBW filtration on $U(\n_{r_{\gamma}}^-) v_{\lambda}$. We construct a polytope $P_{r_{\gamma}}(\lambda) \subset \mathbb{R}_{\geq 0}^{m}, ~ m = \#R_{r_{\gamma}}^-$, whose integral points label a basis of the space $\operatorname{gr} U(\n_{r_{\gamma}}^-) v_{\lambda}$. The polytope $P_{r_{\gamma}}(\lambda)$ is given by three sets of inequalities, namely, inequalities described by Dyck paths, inequalities described by degree paths and inequalities described by paths with coefficients. The inequalities described by degree paths have an upper bound given by PBW degree estimates, hence the upper bound of such an inequality is the maximum PBW degree that can be achieved in $\operatorname{gr} U(\n_{r_{\gamma}}^-) v_{\lambda}$. This degree can be computed as a sum of maximal degrees that can be achieved for each fundamental module (see \cite{CF} Theorem 5.3. (ii) and \cite{BBDF}). The inequalities described by paths with coefficients may be viewed as certain interlaces of degree inequalities and Dyck path inequalities. 
 \begin{exax}
 Let $\g = \s \mathfrak{p}_{6},~ \lambda = m_1 \m_1 + m_2 \m_2 + m_3 \m_3$ where $\m_l$ are the fundamental weights. Let $r_{\alpha_{1,\overline{1}}} = s_1 s_2 s_3 s_2 s_1$. Then the poset $R_{r_{\gamma}}^-$ is given as follows: $ \alpha_1 \rightarrow \alpha_{1,2} \rightarrow \alpha_{1,3}\rightarrow \alpha_{1,\overline{2}}\rightarrow \alpha_{1,\overline{1}} $. The polytope $P_{r_{\gamma}}(\lambda) \subset \mathbb{R}_{\geq 0}^5 $, where we put $s_{\alpha_{1,j}} = s_{1,j}, ~ s_{\alpha_{1,\overline{j}}} = s_{1,\overline{j}}$, is defined by the inequalities 
 \begin{gather*}
     s_{1,1} \leq m_1, \quad s_{1,1} + s_{1,2} \leq m_1 + m_2, \quad s_{1,1} + s_{1,2} + s_{1,\overline{1}} \leq m_1 + m_2 + m_3, \\
    s_{1,1} + s_{1,2} + s_{1,3} + s_{1,\overline{1}} \leq m_1 + m_2 + 2m_3, \quad s_{1,1} + s_{1,2} + s_{1,3} + s_{1,\overline{2}} + s_{1,\overline{1}} \leq m_1 + 2m_2 + 2m_3, \\
    2s_{1,1} + 2s_{1,2} + s_{1,3} + 2s_{1,\overline{1}} \leq 2m_1 + 2m_2 + 2m_3, \\
    2s_{1,1} + s_{1,2} + s_{1,3} + s_{1,\overline{2}} + 2s_{1,\overline{1}} \leq 2m_1 + 2m_2 + 2m_3, \\
    2s_{1,1} + 2s_{1,2} + s_{1,3} + s_{1,\overline{2}} + 2s_{1,\overline{1}} \leq 2m_1 + 3m_2 + 2m_3.
     \end{gather*}
     The first three inequalities are Dyck path inequalities, the fourth and fifth inequalities are degree inequalities and the last three are inequalities with coefficients.
 \end{exax}
 
Let $S_{r_{\gamma}}(\lambda)$ denote the set of integral points in $P_{r_{\gamma}}(\lambda)$. To every point $\textbf{s} = (s_{\alpha})_{\alpha \in R_{r_{\gamma}}^-} \in S_{r_{\gamma}}(\lambda)$, we associate an element 
\[
f^{\textbf{s}}v_{\lambda} = \prod_{\alpha \in R_{r_{\gamma}}^-}f_{\alpha}^{s_{\alpha}} v_{\lambda} \in S(\n_{r_{\gamma}}^-)v_{\lambda}.
\]
Our main result is the following theorem:
\begin{thmx}\label{maintheorem}
  Let $\g$ be a complex simple Lie algebra of type $A, B, C$ or $D$ and let 
  \[
r_{\gamma} = s_{\alpha_i} s_{\alpha_{i+1}} \cdots s_{\alpha_{n-1}} s_{\alpha_n} s_{\alpha_{n-1}} \cdots s_{\alpha_{i+1}} s_{\alpha_i}.
\]
  \begin{enumerate}
      \item For any dominant integral weights $\lambda , \mu$, we have $ S_{r_{\gamma}}(\lambda) + S_{r_{\gamma}}(\mu) = S_{r_{\gamma}}(\lambda + \mu)$ and $P_{r_{\gamma}}(\lambda)$ is a normal polytope.
      \item $S_{r_{\gamma}}(\lambda)$ labels a monomial basis of the PBW graded module $\operatorname{gr} U(\n_{r_{\gamma}}^-) v_{\lambda}$.
      \item $S_{r_{\gamma}}(\lambda)$ labels a monomial basis of $\operatorname{gr} V_{r_{\gamma}}(\lambda)$ and by choosing an order in each factor also of $V_{r_{\gamma}}(\lambda)$.
      \item The annihilating ideal of $\operatorname{gr} U(\n_{r_{\gamma}}^-) v_{\lambda}$ is a monomial ideal.
  \end{enumerate}
\end{thmx}
Furthermore, as shown in Proposition \ref{apoly}, we have that for any subword
\[
r_{\gamma}' = s_{\alpha_k} s_{\alpha_{k+1}} \cdots s_{\alpha_n} \cdots s_{\alpha_i},  \quad k>i
\]
$P_{r_{\gamma}'}(\lambda)$ is a face of $P_{r_{\gamma}}(\lambda)$ defined by setting all the coordinates $s_{\alpha} = 0$ for every $\alpha \notin R_{r_{\gamma}}^-$.

We remark here that unlike the case of triangular subsets in type $A$, our polytopes are in general not faces of the FFLV polytope $P(\lambda)$. For example, if $\g = \s \lie{l}_4$ and $ r_{\alpha_{1,3}} = s_1s_2s_3s_2s_1$, we obtain $R_{r_{\gamma}}^- = \{ \alpha_{1}, \alpha_{1,2}, \alpha_{1,3}, \alpha_{2,3}, \alpha_4 \} $. For $\lambda = m_1 \m_1 + m_2 \m_2 + m_3 \m_3$, the associated polytope is given by the inequalities
\begin{gather*}
    s_{1,1} \leq m_1, ~ s_{1 + s_{1,2}} \leq m_1 + m_2, ~ s_{2,3} + s_{3,3} \leq m_2 + m_3,~ s_{3,3} \leq m_3, \\
    s_{1,1} + s_{1,2} + s_{1,3} + s_{3,3} \leq m_1 + m_2 + m_3,~  s_{1,1} + s_{1,3} + s_{2,3} + s_{3,3} \leq m_1 + m_2 + m_3, \\
    s_{1,1} + s_{1,2} + s_{1,3} + s_{2,3} + s_{3,3} \leq m_1 + 2m_2 + m_3.
\end{gather*}
For the set of integral points, we get for example
\[
S_{r_{\gamma}}(\m_2) = \left\{ \begin{matrix}
0 & & \\
0 & & \\
0 & 0 & 0
\end{matrix} , \quad \begin{matrix}
0 & & \\
1 & & \\
0 & 0 & 0
\end{matrix}, \quad \begin{matrix}
0 & & \\
0 & & \\
1 & 0 & 0
\end{matrix}, \quad \begin{matrix}
0 & & \\
0 & & \\
0 & 1 & 0
\end{matrix}, \quad \begin{matrix}
0 & & \\
1 & & \\
0 & 1 & 0
\end{matrix} \right\} .
\]
Then the last point is not a point in $P(\lambda)$.

Let $\mathbb{U}, \mathfrak{N}, M$ and $v_M$ denote a complex unipotent algebraic group, the corresponding Lie algebra Lie$(\mathbb{U}) = \mathfrak{N}$, a cyclic finite dimensional vector space on which $\mathbb{U}$ acts and the cyclic vector $v_M$ respectively. We have $M = U(\mathfrak{N})v_M$. Fix a total order $\succ $ on the set of generators $\{ f_1 , \ldots , f_N \}$ of $\mathfrak{N}$ and consider the induced homogeneous monomial order on the monomials in $U(\mathfrak{N})$. In \cite{FFLb}, the authors introduced the notion of a favourable module. $M$ is called favourable if it has a monomial basis consisting of essential vectors (see Subsection \ref{linearindependnce} for the definition of essential vectors) labelled by a normal polytope. Because of part (1) and (3) of the theorem, we obtain the following corollary.

\begin{corx}
$V_{r_{\gamma}}(\lambda)$ is a favourable module.
\end{corx}

 The property favourable implies we get certain interesting geometric properties of the corresponding projective varieties, namely, we obtain flat degenerations of Schubert varieties to PBW degenerate varieties and since we have a monomial ideal, this degeneration is a toric degeneration. Moreover, it is both arithmetically Cohen-Macaulay and projectively normal.

The paper is organised as follows: In Section \ref{section1} we give basic definitions. In Sections \ref{section2}, \ref{section3}, \ref{section4} and \ref{section5}, we define the polytopes $P_{r_{\gamma}}(\lambda)$ for type $A, C, B$ and $D$ respectively. In Section \ref{section6}, we prove the Minkowski property and in Section \ref{section7}, we prove the main theorem.

\section{Definitions}\label{section1}
 \subsection{PBW filtrations} 
Let $\g$ be a complex simple Lie algebra and let $\g = \n^+ \oplus \h \oplus \n^-$ be a triangular decomposition. For a dominant integral weight $\lambda$, we denote by $V(\lambda)$ a highest weight module of $\g$ of highest weight $\lambda$ generated by a highest weight vector $v_{\lambda}$. Then we have $V(\lambda) = U(\n^-)v_{\lambda}$ where $U(\n^-)$ is the universal enveloping algebra of $\n^-$. The PBW filtration on $U(\n^- ) $ is defined by
 \[
 U_s(\n^-) := \langle x_{i_1} \cdots x_{i_l} : x_{i_j} \in \n^- , l \leq s \rangle .
 \]
  We obtain an induced filtration on $V(\lambda)$ as follows:
 \[
 V_s(\lambda) := U(\n^-)_sv_{\lambda}.
 \]
 The associated graded module gr $V(\lambda)$ is a module for the abelianised algebra $\g^a = \n^+ \oplus \h \oplus \n^{-,a}$ which is the Lie algebra having the same vector space structure as $\g$ but whose Lie bracket is trivial for the elements of $\n^{-}$. We have gr $V(\lambda) = S(\n^-) v_{\lambda}$ where $S(\n^-)$ is the symmetric algebra of $\n^-$. Then one has $S(\n^-)v_{\lambda} \simeq S(\n^-)/I(\lambda)$ where $I(\lambda) \subset S(\n^-)$ is the annihilating ideal. 
 
 \subsection{Demazure modules} Let $\bb = \n^+ \oplus \mathfrak{h}$ be a Borel subalgebra of $\g$. The weight space $V(\lambda)$ decomposes as a sum of $\h$-weight spaces $V(\lambda)_{\mu}$. The Weyl group $W$ of $\g$ acts on the weights of $V(\lambda)$. For any $w \in W$,  the weight space of weight $w(\lambda)$ generated by $v_{w(\lambda})$ is one-dimensional.

\begin{defn} The \emph{Demazure module} associated to $w \in W$ is defined as  
\[
V_w(\lambda) := U(\bb)v_{w(\lambda)} \subseteq V(\lambda).
\]
\end{defn}

The two extreme cases are: if $w=\operatorname{id}$, then $V_w(\lambda) = \mc v_{\lambda}$ and if $w = w_0$, the longest Weyl group element, then $V_{w_0}(\lambda) = V(\lambda)$. The set of positive roots and negative roots is denoted $R^+$ and $ R^-$ respectively.  To each $w \in  W$, we associate the following subset of roots:
\[
R_w^- := R^+ \cap w^{-1} (R^-).
\]
Let $\n_w^-$ denote the subalgebra of $\n^-$ generated by the root vectors corresponding to the roots in $-R_w^-$, i.e.
\[
\n_{w}^- := \langle f_{\alpha} : \alpha \in R_w^- \rangle \subseteq \n^-.
\] 
 The Weyl group $W$ acts on $U(\g)$ as well as on $V(\lambda)$. We may consider 
\[
w^{-1} (V_w(\lambda))w = w^{-1}(U(\bb)w w^{-1}v_{w(\lambda)})w.
\]

We have $w^{-1}\bb w \subset \n_w^{-} \oplus \n^+$ and $w^{-1} v_{w(\lambda) } w= v_{\lambda}$. Then 
\[
V_w(\lambda) = w(U(\n_w^-) v_{\lambda})w^{-1}
\]
 and in this way, the Demazure module is conjugated to $U(\n_w^-)v_{\lambda}$. We consider the PBW filtration on $U(\n_w^-)v_{\lambda}$ and study the associated graded space $\text{gr }U(\n_w^-)v_{\lambda} \simeq S(\n_w^-)v_{\lambda}$.
 Our main goal is to construct a monomial basis of gr $U(\n_w^-)v_{\lambda}$ labelled by the lattice points of a normal polytope and to give a description of the annihilating ideal $I(\lambda) \subset S(\n_w^-)$.

\subsection{\texorpdfstring{The elements $r_{\gamma}$}{The elements s{gamma}}}
Let $\alpha_i$ denote the simple roots of $\g$. All positive roots of $\g$ are summarised in the table below.

\[
\begin{array}{|c|c|}
\hline
  \operatorname{Type ~ of }~ \g   & \operatorname{Positive ~ roots} \\
  \hline
  A_n   & \alpha_{i,j} = \alpha_i + \cdots + \alpha_j, ~1 \leq i \leq j \leq n \\
  \hline
  B_n & \alpha_{i,j} = \alpha_i + \cdots + \alpha_j, 1 \leq i \leq j \leq n \\
   & \alpha_{i,\bar{j}} = \alpha_i + \cdots + \alpha_{j-1} + 2\alpha_{j} + \cdots + 2\alpha_n, ~2 \leq j \leq n \\
   \hline
   C_n & \alpha_{i,j} = \alpha_i + \cdots + \alpha_j, 1 \leq i \leq j \leq n \\
   & \alpha_{i,\bar{j}} = \alpha_i + \cdots + \alpha_{j-1} + 2\alpha_{j} + \cdots + 2\alpha_{n-1} + \alpha_n, ~1\leq i \leq j \leq n \\
   \hline
   D_n & \alpha_{i,j} = \alpha_i + \cdots + \alpha_j, 1 \leq i \leq j \leq n-1 \\
   & \alpha_{i,\bar{j}} = \alpha_i + \cdots + \alpha_{n-2} + \alpha_{j} + \cdots + \alpha_n, ~1 \leq i < j \leq n \\
   \hline
\end{array}
\]

Let $\gamma$ denote
\begin{itemize}
    \item any positive root $\alpha_{i,j}$ in type $A_n$,
    \item any short root $\alpha_{i,\bar{i}}$ in type $B_n$,
    \item any long root $\alpha_{i,\bar{i}}$ in type $C_n$ and
    \item any root $\alpha_{i,\bar{n}}$ in type $D_n$.
\end{itemize}
The corresponding reflection at the positive root $\gamma$ will be denoted $r_{\gamma}$. Proposition \ref{reflection} is a direct consequence of the following well known fact about reflection groups. For any root $\beta$ of $\g, ~r_{\beta} = v s_{\alpha} v^{-1}$ for any Weyl group element $v$ and simple reflection $s_{\alpha}$ such that $\beta = v(\alpha)$.

\begin{prop}\label{reflection}
$r_{\gamma} = s_{\alpha_i} s_{\alpha_{i+1}} \cdots s_{\alpha_{n-1}} s_{\alpha_n} s_{\alpha_{n-1}} \cdots s_{\alpha_{i+1}} s_{\alpha_i}.$
\end{prop}

Note that in type $A_n$, we have
\[
r_{\alpha_{i,j}} = s_{\alpha_i} s_{\alpha_{i+1}} \cdots s_{\alpha_{j-1}} s_{\alpha_j} s_{\alpha_{j-1}} \cdots s_{\alpha_{i+1}} s_{\alpha_i} \quad \text{for any} \quad \alpha_{i,j}.
\]

To simplify notation, we shall denote $s_{\alpha_i} =: s_i$ for the simple reflections. For the construction of our polytopes and proofs, we will focus on the Weyl group element 
\[
r_{\gamma} = s_1 s_2 \cdots s_{n-1} s_n s_{n-1} \cdots s_2 s_1.
\]
Hence, from now onwards, we set $w := r_{\gamma} = s_1 s_2 \cdots s_{n-1} s_n s_{n-1} \cdots s_2 s_1$ unless stated otherwise. We prove in Proposition \ref{apoly} that our methods easily generalise to other Weyl group elements $s_i s_{i+1} \cdots s_{n-1} s_n s_{n-1} \cdots s_{i+1} s_i$. We will use the following partial order on the set of positive roots,
\[
\alpha_{i_1, j_1} \geq \alpha_{i_2,j_2} \Leftrightarrow i_1 \leq i_2 \wedge j_1 \leq j_2, \quad i_1,~ j_1, ~i_2, ~j_2 \in J
\]
 where $J$ is an indexing set. The partial order on $R_w^-$ is obtained by restricting $\geq$ to the elements of $R_w^-$.
 
\section{Type A}\label{section2} 
Let $\alpha_i, ~\m_i, ~i = 1, \ldots, n$ be the simple roots and fundamental weights of $ \s \mathfrak{l}_{n+1}=\n^+ \oplus \mathfrak{h} \oplus \n^-$. The Killing form on $\mathfrak{h}^*$ is denoted $(\cdot,\cdot)$. In particular, one has $(\m_j, \alpha_i) = \delta_{i,j}$. Recall that all positive roots $\alpha_{i,j}$ of $\s \mathfrak{l}_{n+1}$ satisfy
\[
\alpha_{i,j} = \alpha_i + \alpha_{i+1} + \cdots + \alpha_j, \quad 1 \leq i \leq j \leq n.
\]

For every positive root $\alpha_{i,j}$, we fix a non-zero element $f_{\alpha_{i,j}} \in \n^-$, where $f_{\alpha_{i,j}} = E_{j+1,i}$.

\subsection{Paths and polytopes}
Following \cite{FFLa1}, we would like to define polytopes using paths. Let
\[ 
w = s_{k-l} s_{k-l+1} \ldots s_{k-1} s_{k} s_{k-1} \ldots s_{k-l+1} s_{k-l}
\]
 where $3 \leq k \leq n$ and $1 \leq l \leq k-2$.  The corresponding set of roots and poset structure is given below:
 
 \small
 \[
\begin{matrix}
\alpha_{k-l} & & & & & & & & & \\
\downarrow & & & & & & & & & \\
 \alpha_{k-l,k-l+1} & &  & & & & & & & \\
 \downarrow & & & & & & & & & \\
 \vdots  &   &  &  &  &  & & & &   \\
\downarrow & & & & & & & & & \\
\alpha_{k-l,k-1} &  &   &   &  &  & & & & \\
\downarrow & & & & & & & & & \\
\alpha_{k-l,k} & \rightarrow & \alpha_{k-l+1,k} & \rightarrow & \ldots & \rightarrow & \alpha_{k-1,k} & \rightarrow & \alpha_{k} &  
\end{matrix}
\]
\normalsize

Before defining paths in our poset, we first recall the definition of triangular subsets of the set of positive roots $R^+$ due to \cite{Fou1}. A subset $A \subset R^+$ is called triangular if and only if for all $\alpha_{i_1, j_1} > \alpha_{i_2,j_2}$ with $i_2 \leq j_1 +1$, the root $\alpha_{i_1,j_2} \in A$. Furthermore, if $i_2 \leq j_1$, then the root $\alpha_{i_2,j_1} \in A$.
We use two types of paths to define our polytopes.

\begin{defn} \label{Path} Let $\textbf{p} = (\beta_1 \geq \cdots \geq \beta_s)$ be a sequence of positive roots such that $\beta_i \in R_w^-$.
\begin{itemize}
    \item $\textbf{p}$ is called a \emph{Dyck path} if $\{ \beta_1, \ldots , \beta_s \}$ is a triangular subset of $R^+$.
    \item $\textbf{p}$ is called a \emph{degree path} if $\{ \beta_1, \ldots , \beta_s \}$ is not a triangular subset of $R^+$.
\end{itemize}
\end{defn}

We consider the following Dyck paths:
\begin{gather*}
    (\alpha_{k-l}, \ldots , \alpha_{k-l,j}),  \quad k-l \leq j \leq k-1; \quad (\alpha_{k-l}, \ldots , \alpha_{k-l,k} , \alpha_k) ; \\
    (\alpha_{i,k} , \ldots , \alpha_{k}), \quad k-l+1 \leq i \leq k ; \quad (\alpha_{k-l}, \alpha_{k-l,k}, \ldots , \alpha_k) ; \\
    (\alpha_{k-l}, \alpha_{k-l,k-l+1}, \ldots, \alpha_{k-l,j}, \alpha_{i,k}, \alpha_{i+1,k}, \ldots, \alpha_k), \quad j \in \{ k-l , \ldots , t , k \}, \quad i \in \{t+1, \ldots , k \},
\end{gather*}

where $k-l+1 \leq t \leq k-2$. For the degree paths, we make the following choice:
\begin{equation}\label{extrapath}
(\alpha_{k-l}, \ldots , \alpha_{k-l,r}, \alpha_{k-l,k}, \alpha_{i,k}, \ldots , \alpha_k), \text{ for } k-l+1 \leq i \leq k-1, \quad i \leq r \leq k-1.
\end{equation}

Let $\lambda = \sum_{i=1}^n m_i \m_i$ be a dominant integral weight. To every Dyck path $\textbf{p} = (\beta_{1} , \ldots , \beta_s )$, we associate an inequality
\begin{equation}\label{ineq1}
s_{\beta_1} + \cdots s_{\beta_s} \leq m_i + m_{i+1} + \cdots + m_{q},
\end{equation}

where $\beta_1 = \alpha_{i,j}$ and $\beta_s = \alpha_{p,q}$. To every degree path \ref{extrapath}, we associate an inequality
\begin{equation}\label{ineq-2}
 s_{k-l} + \ldots + s_{k-l,r} + s_{k-l,k} + s_{i,k} + \ldots + s_k \leq m_{k-l} + \ldots + m_k + \sum_{j=i}^r m_j
\end{equation}

where $k-l+1 \leq i \leq k-1, \quad i \leq r \leq k-1$.

\begin{defn}
 Let $P_w(\lambda) \subset \mathbb{R}_{\geq0}^{\# R_w^-}$ be the polytope given by the inequalities \ref{ineq1} and \ref{ineq-2}. 
\end{defn}

\begin{rem}
We note that subsets of our choice of Dyck paths and degree paths also satisfy the conditions of Definition \ref{Path} and one can write down an inequality corresponding to these paths. However, this inequality would be redundant for the definition of the polytope. 
\end{rem}

\begin{rem} \label{idealremmark}
For a vector $f_{\beta_1}^{a_1} \cdots f_{\beta_s}^{a_s}v_{\lambda} \in \operatorname{gr} V(\lambda) $, the PBW-degree of the vector is defined as the sum of all the exponents $a_i$. By degree of $\operatorname{gr} V(\lambda)$, we mean the maximal PBW-degree that can be achieved in $\operatorname{gr} V(\lambda)$. In \cite{CF}, see Theorem 5.3. (ii), it is proved that for any dominant integral weights $\lambda , \mu$, $\deg ( \operatorname{gr} V(\lambda + \mu)) =  \deg(\operatorname{gr}V(\lambda)) +\deg (\operatorname{gr} V(\mu))$. We will frequently use the following fact: for any dominant integral weight $\lambda = \sum_{j=1}^n m_j \m_j $,
\[
\deg (\operatorname{gr} V(\lambda)) = \sum_{j=1}^n \left( \underbrace{ \deg (\operatorname{gr} V(\m_j)) + \cdots +  \deg (\operatorname{gr} V(\m_j))}_{m_j \text{-times}} \right) .
\]
Recall the notation $U(\n_w^-)$ for the universal enveloping algebra of $\n_w^- := \langle f_{\alpha} : \alpha \in R_w^- \rangle \subseteq \n^-$. Looking back at our poset $R_w^-$, one sees immediately that
\[
\deg (S(\n_w^-)(m_d \m_d)) = \begin{cases} 2m_d  & \text{for } k-l+1 \leq d \leq k-1, \\
m_d & \text{otherwise.}
\end{cases}
\]
Set $s_{\alpha} = 0$ if $\alpha \notin \textbf{p}$. Then a straightforward computation shows that the right hand side of the inequalities \ref{ineq-2} is the maximum PBW-degree that can be achieved for any multi-exponent $(s_{\alpha})_{\alpha \in R_w^-}$.
\end{rem}

\begin{rem}
Let $w = w_0$, the longest Weyl group element so that $R_w^- = R^+$. Let $\lambda = \m_d$. We have $ V(\m_d) = \bigwedge^d \mc^{n+1}$. The non-trivial operators $f_{i,j}:=f_{\alpha_{i,j}}$ act on $\mc^{n+1}$ with standard basis $(e_i)_{i=1, \ldots, n+1}$ by the usual formula $f_{i,j}e_l = \delta_{i,l} e_{j+1}$. The highest weight vector $v_{\m_d}$ can be chosen to be $e_1 \wedge \cdots \wedge e_d$. Now consider the associated graded module $\operatorname{gr} V(\lambda) = S(\n^-)(e_1 \wedge \cdots \wedge e_d)$. The operators $f_{i,j}$ act trivially on $e_1 \wedge \cdots \wedge e_d$ unless $(\m_d, \alpha_{i,j}) =1$ (or equivalently $i \leq d \leq j$). Suppose $f_{i_1,j_1},~f_{i_1,j_2},~f_{i_2,j_1},~f_{i_2,j_2} \in S(\n^-)$ act non-trivially on the highest weight vector. Then the vectors $f_{i_1,j_1}f_{i_2,j_2} e_1 \wedge \cdots \wedge e_d$ and $f_{i_1,j_2} f_{i_2,j_1} e_1 \wedge \cdots \wedge e_d$ are linearly dependent in $S(\n^-)(e_1 \wedge \cdots \wedge e_d)$. For constructing the basis, we only need to choose one of them. The definition of the FFLV polytope \cite{FFLa1} implies the latter choice. In our case, we only have one choice, namely the former since $\alpha_{i_2,j_1} \notin R_w^-$. Therefore we must impose this on the inequalities defining our polytope. 
\end{rem}

In the following, let $S_w(\lambda)$ and $R_i$ denote the subsets
\[
S_w(\lambda) = P_w(\lambda) \cap \mathbb{Z}_{\geq0}^{\#R_w^-} \quad \text{and} \quad R_i = \{ \alpha \in R_w^- | (\m_i , \alpha) = 1 \}.
\]
For every $\alpha \in R_i$, we denote by $\textbf{m}_{\alpha}$ the multi-exponent $(m_{\alpha})_{\alpha \in R_w^-} \in \mathbb{Z}_{\geq0}^{\#R_w^-}$ defined by $m_{\alpha} = 1$ and zero otherwise. We end this section with the following corollary that follows immediately from the definition of $P_w(\lambda)$.

\begin{cor}\label{points2}
The following points are in $S_w(\m_i)$:
\begin{itemize}
    \item $\textbf{m}_{\alpha}$ for every $\alpha \in R_i$;
    \item $\textbf{m}_{\alpha} + \textbf{m}_{\alpha'}$ where $\alpha = \alpha_{k-l,p},~\alpha' = \alpha_{q,k},~ i \leq p \leq k-1,~2 \leq q \leq i$, for every $\alpha, \alpha' \in R_i$. 
\end{itemize}
\end{cor}

 \section{Type C}\label{section3}
Let $\alpha_i, ~\m_i, ~i=1, \ldots, n$ be the simple roots and fundamental weights of  $ \s \mathfrak{p}_{2n} = \n^+ \oplus \mathfrak{h} \oplus \n^-$. The killing form on $\mathfrak{h}^*$ is denoted $(\cdot, \cdot)$. The positive roots of $\g $ can be divided into two groups as follows:
\begin{gather*}
    \alpha_{i,j} = \alpha_i + \alpha_{i+1} + \cdots + \alpha_j, 1 \leq i \leq j \leq n, \\
    \alpha_{i,\overline{j}} = \alpha_i + \alpha_{i+1} + \cdots + \alpha_{j-1} + 2 \alpha_{j} \cdots + 2 \alpha_{n-1} + \alpha_n, 1 \leq i \leq j \leq n.
\end{gather*}
Note that we also have the equality $\alpha_{1,n} = \alpha_{1,\overline{n}}$. Let $w = s_1s_2 \cdots s_n \cdots s_2 s_1$, then the root poset of the roots in $R_w^-$ is given by
\[
  \alpha_1 \rightarrow \alpha_{1,2} \rightarrow \ldots \rightarrow \alpha_{1,n} \rightarrow \alpha_{1,\overline{n-1}} \rightarrow \cdots \rightarrow \alpha_{1,\overline{2}} \rightarrow \alpha_{1,\overline{1}} 
\]

 For every positive root $\alpha$, we fix a non-zero element $f_{\alpha} \in \n^-$. The elements $f_{\alpha}$ are given by
 \begin{gather*}
     f_{i,j} = E_{j+1,i} - E_{n+j+1,n+i}, \quad 1 \leq i \leq j \leq n, \\
      f_{i,\overline{j}} = E_{n+i,j} + E_{n+j,i}, \quad 1 \leq i < j \leq n, \\
      f_{i, \overline{i}} = E_{n+i,i}
 \end{gather*}
 
Next, we define a total order on the generators $f_{\alpha}, \alpha \in R_w^-$ of $S(\n_w^-)$ as follows:
 \[
 f_{1,\overline{1}} \succ f_{1,\overline{2}} \succ \cdots \succ f_{1,n} \succ \cdots \succ f_{1}.
 \]  The induced graded lexicographical order on the monomials will be denoted by the same symbol. (Note that the total order $\succ$ is defined in \cite{FFLa2} on the generators of $S(\n^-)$. Here we restrict to the generators of the subalgebra).  
 
  As before we define polytopes using paths. We have the following version of paths for type $C$.
 
 \begin{defn}\label{Cnpaths} A \emph{Dyck path} is a sequence of roots of the form
 \begin{gather}
     (\alpha_1, \ldots , \alpha_{1,j}),~ 1 \leq j \leq n-1, \quad (\alpha_1, \ldots , \alpha_{1,n-1}, \alpha_{1,\overline{1}}).
 \end{gather}
 A \emph{degree path} is a sequence of roots of the form
 \begin{equation}\label{degreepathc}
     (\alpha_1 , \ldots , \alpha_{1, \overline{j}} , \alpha_{1,\overline{1}}),~ 2 \leq j \leq n.
 \end{equation}
 \end{defn}
 
Let $\lambda = \sum_{j=1}^{n} m_j \m_j $ be an $\s \mathfrak{p}_{n}$-dominant integral weight. To the paths defined above, we associate the following inequalities respectively:
 \begin{gather} \label{ineqc}
      s_{1,1} + \cdots + s_{1,j} \leq m_1 + \cdots + m_j, \quad s_{1,1} + \cdots + s_{1,n-1} + s_{1,\overline{1}} \leq m_1 + \cdots + m_n \\
    s_{1,1} + \cdots + s_{1, \overline{j}} + s_{1,\overline{1}} \leq m_1 + \cdots + 2m_j + \cdots + 2m_n.
 \end{gather}
 
Before we can define the polytopes $P_w(\lambda)$, we need one more set of inequalities, namely, inequalities with coefficients. 

\begin{defn}\label{coefficientsc}
 Let $\textbf{p}$ be a degree path as in Definition \ref{Cnpaths}. We associate to $\textbf{p}$ tuples $t_{\textbf{p}}^k$ labelled by a number $k \in \{ j-1, j, \ldots, n-1 \}$. For every number $k$, the tuple $t_{\textbf{p}}^k$ is defined as follows: we set $t_{\alpha_{1,1}} = \cdots = t_{\alpha_{1,k}} = t_{\alpha_{1,\overline{1}}} = 2$ and $t_{\alpha} = 1$ otherwise. We call the elements of each tuple $t_{\textbf{p}}^k$ \emph{coefficients} of the degree path $\textbf{p}$. The assignment of the coefficients is done component-wise, i.e. $\alpha \in \textbf{p}$ is assigned the coefficient $t_{\alpha} \in t_{\textbf{p}}^k$. The resulting path is called a \emph{path with coefficients}. 
\end{defn}

Let $q_{\textbf{p}}^{\lambda}$ denote the right hand side of the inequality given by a path $\textbf{p}$. To every path with coefficients, we associate an inequality
\begin{equation}\label{inequalitywithcoefficients}
    2s_{1,1} + \cdots + 2 s_{1,k} + s_{1,k+1} + \ldots + s_{1,\overline{j}} + 2s_{1,\overline{1}} \leq  q_{\textbf{p}}^{\lambda} + \sum_{i=1}^k m_i.
\end{equation}

\begin{defn}
 Let $P_w(\lambda)$ be the polytope given by the inequalities \ref{ineqc}- \ref{inequalitywithcoefficients}.
\end{defn}
 
\begin{exa}
In Table \ref{coefficientssp4}, we present a complete list of paths with coefficients associated to every degree path for $\g = \s \mathfrak{p}_8$. We also give for every path with coefficients the corresponding inequality.
\end{exa}

\begin{table}
    \centering
     \resizebox{\textwidth}{!}{%
   $ \begin{array}{|c|c|c|}
   \hline
    \text{Degree path}     & \text{Coefficients} & \text{Inequality} \\
    \hline
      ( \alpha_1, \alpha_{1,2}, \alpha_{1,3}, \alpha_{1,4}, \alpha_{1,\bar{1}})   & (2,2,2,1,2) & 2s_{1,1} + 2s_{1,2} + 2s_{1,3} + s_{1,4} + 2s_{1,\overline{1}} \leq q_{\textbf{p}}^{\lambda} + \sum_{i=1}^3 m_i \\
     \hline
     ( \alpha_1, \alpha_{1,2}, \alpha_{1,3}, \alpha_{1,4},  & (2,2,2,1,1,2) & 2s_{1,1} + 2s_{1,2} + 2s_{1,3} + s_{1,4} + s_{1,\overline{3}} + 2s_{1,\overline{1}} \leq q_{\textbf{p}}^{\lambda} + \sum_{i=1}^3 m_i \\
     \alpha_{1,\overline{3}}, \alpha_{1,\overline{1}}) & (2,2,1,1,1,2) & 2s_{1,1} + 2s_{1,2} + s_{1,3} + s_{1,4} + s_{1,\overline{3}} + 2s_{1,\overline{1}} \leq q_{\textbf{p}}^{\lambda} + \sum_{i=1}^2 m_i \\
      \hline
       (\alpha_1, \alpha_{1,2}, \alpha_{1,3}, \alpha_{1,4},  & (2,2,2,1,1,1,2) & 2s_{1,1} + 2s_{1,2} + 2s_{1,3} + s_{1,4} + s_{1,\overline{3}}+ s_{1,\overline{2}} + 2s_{1,\overline{1}} \leq q_{\textbf{p}}^{\lambda} + \sum_{i=1}^3 m_i \\
     \alpha_{1,\overline{3}}, \alpha_{1,\overline{2}} \alpha_{1,\overline{1}} ) & (2,2,1,1,1,1,2) & 2s_{1,1} + 2s_{1,2} + s_{1,3} + s_{1,4} + s_{1,\overline{3}}+ s_{1,\overline{2}} + 2s_{1,\overline{1}} \leq q_{\textbf{p}}^{\lambda} + \sum_{i=1}^2 m_i \\
      &  (2,1,1,1,1,1,2) & 2s_{1,1} + s_{1,2} + s_{1,3} + s_{1,4} + s_{1,\overline{3}}+ s_{1,\overline{2}} + 2s_{1,\overline{1}} \leq q_{\textbf{p}}^{\lambda} + m_1 \\
      \hline
    \end{array}$
    }
    \caption{Coefficients and inequalities for $\s \mathfrak{p}_8$.}
    \label{coefficientssp4}
\end{table}
 
 \subsection{Fundamental sets} In this subsection, we study the case when $\lambda = \m_i, 1 \leq i \leq n$. Table \ref{polytopesforfundamentalmodules} summarises the polytopes for fundamental modules.  
 
 \begin{table}
      \centering
 \begin{equation*}
     \begin{array}{|c|c|}
     \hline
   \text{Fundamental~weight}   & \text{Polytope~inequalities}  \\
   \hline
     \m_1  & s_{1,1} + s_{1,2} + \cdots + s_{1,\overline{1}} \leq 1 \\
    \hline 
    \m_i , ~1 < i < n & s_{1,1} + \cdots + s_{1,i-1} = 0 \\
                     & s_{1,i} + \cdots + s_{1,\overline{i+1}} + s_{1,\overline{1}} \leq 1 \\
                     & s_{1,i} + \cdots  + s_{1,\overline{1}} \leq 2  \\
                    & s_{1,i} + \cdots  + 2s_{1,\overline{1}} \leq 2  \\
                    \hline
   \m_n & s_{1,1} + \cdots + s_{1,n-1} = 0 \\
  & s_{1,n} + s_{1, \overline{n-1}} \cdots  + s_{1,\overline{1}} \leq 1 \\
        & s_{1,n} + \cdots + s_{1,\overline{1}} \leq 2  \\
        & s_{1,n} + \cdots + 2s_{1,\overline{1}} \leq 2  \\
        \hline
 \end{array}
 \end{equation*}
     \caption{Polytopes for fundamental modules of $\s \mathfrak{p}_{2n}$.}
     \label{polytopesforfundamentalmodules}
 \end{table}
 
 Let $S_w(\lambda)$ and $R_i$ denote the sets
 \[
S_w(\lambda) = P_w(\lambda) \cap \mathbb{Z}_{\geq0}^{\#R_w^-} \quad R_i = \{ \alpha \in R_w^- : (\m_i , \alpha) \neq 0 \}.
 \]
 For every $\alpha \in R_i$, we denote by $\textbf{m}_{\alpha}$ the multi-exponent $(m_{\alpha})_{\alpha \in R_w^-} \in \mathbb{Z}_{\geq0}^{\#R_w^-}$ defined by $m_{\alpha} = 1$ and zero otherwise. The following lemma follows immediately from the definition of $P_w(\lambda)$.
 
 \begin{lem}\label{fundamentalpoints} 
 Let $\lambda = \m_i$, then the following points are in $S_w(\m_i)$:
\begin{itemize}
        \item[i)] $\textbf{m}_{\alpha}$ for all $\alpha \in R_i$.
        \item[ii)] $\textbf{m}_{\alpha} + \textbf{m}_{\alpha'}$ for all  $\alpha = \alpha_{1,j},~ i \leq j \leq n,~ \alpha' = \alpha_{1, \overline{j'}},~ 2 \leq j' \leq i$  (or $\alpha = \alpha_{1,\overline{j}},~  \alpha' = \alpha_{1,\overline{j'}},~ 2 \leq j \leq j' \leq i$).
    \end{itemize}
 \end{lem}
 
 \section{Type B}\label{section4}
 Let $\alpha_i, ~\m_i, ~i=1, \ldots, n$ denote the simple roots and fundamental weights of $ \s \mathfrak{o}(2n+1) = \n^+ \oplus \mathfrak{h} \oplus \n^-$. The killing form on $\mathfrak{h}^*$ is denoted $(\cdot, \cdot)$. The positive roots of $\s \mathfrak{o}(2n+1)$ can be divided into two groups as follows:
 \begin{gather}
     \alpha_{i,j} = \alpha_i + \cdots + \alpha_j , ~1 \leq i \leq j \leq n , \\
     \alpha_{i , \overline{j}} = \alpha_i + \cdots + \alpha_{j-1} + 2 \alpha_j + \cdots + 2 \alpha_n, ~2 \leq j \leq n.
 \end{gather}
 For $w = s_1s_2 \cdots s_n \cdots s_2 s_1$ the root poset of the roots in $R_w^-$ is given by
\[
 \alpha_1 \rightarrow \alpha_{1,2} \rightarrow \ldots \rightarrow \alpha_{1,n} \rightarrow \alpha_{1,\overline{n}} \rightarrow \alpha_{1,\overline{n-1}} \rightarrow \ldots \rightarrow \alpha_{1,\overline{3}}\rightarrow \alpha_{1,\overline{2}} 
\]
 For every positive root $\alpha$, we fix a non-zero element $f_{\alpha} \in \n^-$. Explicitly, the elements $f_{\alpha}$ are given by 
 \begin{gather}
     f_{i,j} = E_{j+1,i} - E_{n+j+1,n+i}, \quad 1 \leq i \leq j \leq n-1, \\
     f_{i,n} = E_{n+i,2n+1} - E_{2n+1,i}, \quad 1 \leq i \leq n , \\
     f_{i,\overline{j}} = E_{n+i,j} - E_{n+j,i}.
 \end{gather}
  Note that we also have the operator $ \bar{f}_{i,n} = f_{i,n}^2 = -E_{n+i,i} $ where the square corresponds to the multiplication in the universal enveloping algebra. Hence to each positive root $\alpha_{i,n}$, we attach two operators $f_{i,n}$ and $\bar{f}_{i,n}$. We order the generators of $S(\n^-)$ by
 \begin{gather*}
     f_{n,n} \succ \\
     f_{n-1,\overline{n}} \succ f_{n-1,n} \succ f_{n-1, n-1} \succ \\
     f_{n-2,\overline{n-1}} \succ f_{n-2, \overline{n}} \succ f_{n-2, n} \succ f_{n-2,n-1} \succ f_{n-2, n-2} \succ \\
     \cdots \succ \cdots \succ \cdots \succ \\
     f_{1,\overline{2}} \succ f_{1,\overline{3}} \succ \cdots \succ f_{1,\overline{n}} \succ f_{1,n} \succ f_{1,n-1} \succ \cdots \succ f_{1,2} \succ f_{1,1}.
 \end{gather*}
 We denote by the same symbol $\succ$ the induced homogeneous lexicographical order. By restriction, we obtain a total order on the generators of $S(\n_w^-)$.
 
 \subsection{Polytopes} Let $\lambda = \sum_{j=1}^{n} m_j \m_j $ be an $\s \mathfrak{o}(2n+1)$-dominant integral weight. As before, we would like to define polytopes using inequalities arising  from paths.

\begin{defn} \label{pathsb}
A path in the poset $(R_w^-, \leq)$ is one of the following sequence of roots:
\begin{itemize}
    \item[i)] a \emph{Dyck path} is a sequence of roots of the form $(\alpha_1, \ldots, \alpha_{1,j})$ where $j=1, \ldots, n-1$,
    \item[ii)] a \emph{degree path} is a sequence of roots of the form $(\alpha_{1} , \ldots, \alpha_{1,\overline{j}})$ where $j=2, \ldots,n$.
\end{itemize}
\end{defn}

To the paths i) and ii) above, we associate the following inequalities respectively:
\begin{gather}\label{typebineq}
    s_{1,1} + \cdots + s_{1,j} \leq m_1 + \cdots + m_j, \\
    s_{1,1} + \cdots + s_{1,\overline{j}} \leq 2m_1 + m_2 + \cdots + m_{j-1} + 2 m_j + \cdots + 2m_{n-1} + m_n.
\end{gather}

 Similar to the type $C$ case, we need one more set of inequalities to define the polytopes $P_w(\lambda)$, namely, the inequalities arising from paths with coefficients. The description of the coefficients is similar to the one given in Section \ref{section3} for type $C$. 
 
 \begin{defn}
 Let $\textbf{p} = (\alpha_{1} , \ldots, \alpha_{1,\overline{j}})$ be a degree path as in Definition \ref{pathsb}. We associate to $\textbf{p}$ tuples $t_{\textbf{p}}^k$ labelled by a number $k \in \{ 0, j-1, \ldots, n-1 \}$. For every number $k$, the tuple $t_{\textbf{p}}^k$ is defined as follows:
 \begin{itemize}
     \item If $k=0$, then we put $t_{\alpha_{1,1}} = \cdots = t_{\alpha_{1,n-1}} = t_{\alpha_{1,\bar{j}}} = \cdots = t_{\alpha_{1,\bar{n}}} = 2, ~t_{\alpha_{1,n}} = 1$.
     \item If $k \neq 0$, then we put $t_{\alpha_{1,1}} = \cdots = t_{\alpha_{1,k}} = 2$ and $t_{\alpha} = 1$ otherwise.
 \end{itemize}
 We call the elements of $t_{\textbf{p}}^k$ coefficients of $\textbf{p}$. The assignment of the coefficients is done component-wise. The resulting path is called a \emph{path with coefficients}.
 \end{defn}

Let $q_{\textbf{p}}^{\lambda}$ denote the right hand-side of the inequality given by the path $\textbf{p}$. To every path with coefficients, we associate an inequality
\begin{equation}\label{typebineq2}
    2\sum_{l=1}^k s_{1,l} + s_{1,k+1} + \cdots + s_{1,\overline{j}} \leq q_{\textbf{p}}^{\lambda} + \sum_{l=2}^{k} m_l
\end{equation}

 if the tuple of coefficients is $t_{\textbf{p}}^k, ~k \neq0$ and
\begin{equation}\label{typebineq3}
    2\sum_{l=1}^{n-1} s_{1,l} + s_{1,n} + 2\sum_{l=j}^n s_{1,\overline{l}} \leq q_{\textbf{p}}^{\lambda} + \sum_{l=2}^{n} m_l + \sum_{l=j}^{n-1} m_l.
\end{equation}
if the tuple of coefficients is $t_{\textbf{p}}^0$.

\begin{defn}
 The polytope $P_w(\lambda)$ is defined by the inequalities \ref{typebineq}- \ref{typebineq3}. 
\end{defn}

\begin{exa}
In Table \ref{coefficientss09}, we present a complete list of paths with coefficients associated to every degree path for $\g = \s \mathfrak{o}_9$. We also give for every path with coefficients the corresponding inequality.
\end{exa}

\begin{table}[ht]
\centering
\resizebox{\textwidth}{!}{%
   $ \begin{array}{|c|c|c|}
   \hline
    \text{Degree~path}  & \text{Coefficients} & \text{Inequality} \\
    \hline
     ( \alpha_1, \alpha_{1,2}, \alpha_{1,3}, \alpha_{1,4},    & (2,2,2,1,2) & 2s_{1,1} + 2s_{1,2} + 2s_{1,3} + s_{1,4} + 2s_{1,\overline{4}} \leq q_{\textbf{p}}^{\lambda} + \sum_{l=2}^4 m_l  \\
    \alpha_{1,\overline{4}}) & (2,2,2,1,1) & 2s_{1,1} + 2s_{1,2} + 2s_{1,3} + s_{1,4} + s_{1,\overline{4}} \leq q_{\textbf{p}}^{\lambda} + \sum_{l=2}^3 m_l  \\
     \hline
     ( \alpha_1, \alpha_{1,2}, \alpha_{1,3}, \alpha_{1,4},  & (2,2,2,1,2,2) & 2s_{1,1} + 2s_{1,2} + 2s_{1,3} + s_{1,4} + 2s_{1,\overline{3}} + 2s_{1,\overline{4}} \leq q_{\textbf{p}}^{\lambda} + \sum_{l=2}^4 m_l  + m_3 \\
     \alpha_{1,\overline{4}}, \alpha_{1,\overline{3}}) & (2,2,2,1,1,1) & 2s_{1,1} + 2s_{1,2} + 2s_{1,3} + s_{1,4} + s_{1,\overline{3}} + s_{1,\overline{4}} \leq q_{\textbf{p}}^{\lambda} + \sum_{l=2}^3 m_l \\
      & (2,2,1,1,1,1) & 2s_{1,1} + 2s_{1,2} + s_{1,3} + s_{1,4} + s_{1,\overline{3}} + s_{1,\overline{4}} \leq q_{\textbf{p}}^{\lambda} +  m_2 \\
      \hline
      ( \alpha_1, \alpha_{1,2}, \alpha_{1,3}, \alpha_{1,4},  & (2,2,2,1,2,2,2) &  2s_{1,1} + 2s_{1,2} + 2s_{1,3} + s_{1,4} + 2s_{1,\overline{4}} + 2s_{1,\overline{3}} + 2s_{1,\overline{2}} \leq q_{\textbf{p}}^{\lambda} + \sum_{l=2}^4 m_l  + \sum_{l=2}^3 \\
     \alpha_{1,\overline{4}}, \alpha_{1,\overline{3}} \alpha_{1,\overline{2}} ) & (2,2,2,1,1,1,1) & 2s_{1,1} + 2s_{1,2} + 2s_{1,3} + s_{1,4} + s_{1,\overline{4}} + s_{1,\overline{3}} + s_{1,\overline{2}} \leq q_{\textbf{p}}^{\lambda} + \sum_{l=2}^3 m_l   \\
      &  (2,2,1,1,1,1,1) & 2s_{1,1} + 2s_{1,2} + s_{1,3} + s_{1,4} + s_{1,\overline{4}} + s_{1,\overline{3}} + s_{1,\overline{2}} \leq q_{\textbf{p}}^{\lambda} + m_2 \\
      & (2,1,1,1,1,1,1) &  2s_{1,1} + s_{1,2} + s_{1,3} + s_{1,4} + s_{1,\overline{4}} + s_{1,\overline{3}} + s_{1,\overline{2}} \leq q_{\textbf{p}}^{\lambda} \\
      \hline
    \end{array}$%
}
\caption{Coefficients and inequalities for $\s \mathfrak{o}_{9}$}
    \label{coefficientss09}
\end{table}

 \subsection{Fundamental sets} We study the case $\lambda = \m_i, ~i=1, \ldots, n$. In Table \ref{fundamentalb}, we summarise the polytopes for fundamental modules.
 
\begin{table}[ht]
    \centering
    \begin{equation*}
     \begin{array}{|c|c|}
     \hline
   \text{Fundamental weight}   & \text{Polytope}  \\
   \hline 
   \m_i, ~ 1 \leq i \leq n-1   & s_{1,1} + s_{1,2} + \cdots + s_{1,i-1} = 0 \\
    & s_{1,i} + \cdots + s_{1,n-1} + s_{1,\overline{n}} + \cdots + s_{1,\overline{i+1}}  \leq 1 \\
   &  s_{1,i} + \cdots +  s_{1,\overline{2}} \leq 2 \\
    &  2s_{1,i} + \cdots + 2s_{1,n-1} + s_{1,n} + 2 s_{1,\overline{n}} + \cdots + 2s_{1,\overline{i+1}} \leq 2 \\
    \hline
    \m_n & s_{1,1} + \cdots + s_{n-1} = 0 \\
    & s_{1,n} + \cdots + s_{1,\overline{2}}  \leq 1 \\
    \hline
 \end{array}
\end{equation*}
    \caption{Polytopes for fundamental modules of $\s \mathfrak{o}_{2n+1}$.}
    \label{fundamentalb}
\end{table}
 
As before, we denote by $S_w(\lambda)$ and $R_i$ the sets
 \[
S_w(\lambda) = P_w(\lambda) \cap \mathbb{Z}_{\geq0}^{\#R_w^-} \quad R_i = \{ \alpha \in R_w^- : (\m_i , \alpha) \neq 0 \}.
 \]
 For every $\alpha \in R_i$, we denote by $\textbf{m}_{\alpha}$ the multi-exponent $(m_{\alpha})_{\alpha \in R_w^-} \in \mathbb{Z}_{\geq0}^{\#R_w^-}$ defined by $m_{\alpha} = 1$ and zero otherwise. The following lemma follows immediately from the definition of $P_w(\m_i)$.
 
\begin{lem}\label{fundamentalpointsB}The following points are in $S_w(\m_i)$:
\begin{itemize}
        \item[i)] the points $\textbf{m}_{\alpha}$  for all $\alpha \in R_i$,
        \item[ii)] the points $\textbf{m}_{\alpha} + \textbf{m}_{\alpha'}$ such that $\alpha = \alpha_{1,j},~ \alpha' = \alpha_{1, \overline{j'}}, ~ i \leq j \leq n, ~ 2 \leq j' \leq n$ (or $\alpha = \alpha_{1,\overline{j}}, ~ \alpha' = \alpha_{1,\overline{j'}}, ~ 2 \leq j \leq j' \leq i$) for all $\alpha, \alpha' \in R_i$,
        \item[iii)] the point $2 \textbf{m}_{\alpha}, ~ \alpha = \alpha_{1,n}$.
    \end{itemize}
 \end{lem}

 \section{Type D}\label{section5}
 Let $\alpha_i, ~\m_i, ~i=1, \ldots ,n$ denote the simple roots and fundamental weights of $ \s \mathfrak{o}(2n) =\n^+ \oplus \mathfrak{h} \oplus \n^-$. The killing form on $\mathfrak{h}^*$ is denoted $(\cdot, \cdot)$.  The positive roots of $\s \mathfrak{o}(2n)$ can be divided into two groups as follows:
 \begin{gather}
     \alpha_{i,j} = \alpha_i + \cdots + \alpha_j , \quad 1 \leq i \leq j \leq n-1 ,\\
     \alpha_{i,\overline{j}} = \alpha_i \cdots + \alpha_{n-2} + \alpha_j + \cdots + \alpha_{n}, \quad 1 \leq i < j \leq n. 
 \end{gather}
 Note that $\alpha_{i, \overline{n}} = \alpha_i + \cdots + \alpha_{n-2} + \alpha_{n}$ and $\alpha_{n-1, \overline{n}} = \alpha_n$. We consider the word $w = s_1 s_2 \cdots s_{n-1} s_n \hat{s_{n-1}} s_{n-2} \cdots s_{1}$ where the hat means $s_{n-1}$ is omitted. The root poset $R_w^-$ is given by

\begin{center}
\resizebox{\columnwidth}{!}{%
    \begin{tikzpicture}
    \node (a) at (-9,1) {$\alpha_{1}$};
    \node (b) at (-7,1) {$\alpha_{1,2}$};
    \node (c) at (-5,1) {$\cdots$};
    \node (d) at (-3,1) {$\alpha_{1,n-3}$};
    \node (e) at (-1,1) {$\alpha_{1,n-2}$};
    \node (f) at (1,1) {$\alpha_{1,n-1}$};
    \node (g) at (2,0) {$\alpha_{1,\overline{n-1}}$};
    \node (h) at (4,0) {$\alpha_{1,\overline{n-2}}$};
    \node (i) at (6,0) {$\cdots$};
    \node (j) at (8,0) {$\alpha_{1,\overline{3}}$};
    \node (k) at (10,0) {$\alpha_{1,\overline{2}}$};
    \node (l) at (0,0) {$\alpha_{1,\overline{n}}$};
    \draw[->] (a) -- (b);
    \draw[->] (b)-- (c);
    \draw[->] (c)-- (d);
    \draw[->] (d)-- (e);
    \draw[->] (e)-- (f);
    \draw[->] (f)-- (g);
    \draw[->] (g)-- (h);
    \draw[->] (h)-- (i);
    \draw[->] (i) -- (j);
    \draw[->] (j)-- (k);
    \draw[->] (e) -- (l);
    \draw[->] (l)-- (g);
    \end{tikzpicture}
    }
\end{center}

\begin{rem}
Before we proceed, we would like to remark here that for the word $w' = s_1 s_2 \cdots s_{n-1} s_n s_{n-1} s_{n-2} \cdots s_{1}$, the construction resembles the situation in type $A$. In particular, one obtains 
\[
R_{w'}^- = \{ \alpha_1 , \ldots , \alpha_{1,n-2} , \alpha_{1,\overline{n}} , \alpha_{2,\overline{n}} , \ldots , \alpha_{n-1,\overline{n}} \}
\]
and the polytope $P_{w'}(\lambda)$ where $\lambda = \sum_{i=1}^n m_i \m_i$ is defined by the Dyck path inequalities
\begin{gather*}
    \sum_{l=1}^j s_{1,l} \leq \sum_{l=1}^j m_l, \quad 1 \leq j \leq n-2 \\
    \sum_{l=i}^{n-1} s_{i,\overline{n}}  \leq m_i + \cdots + m_{n-2} + m_n , \quad 2 \leq i \leq n-1 \\
    \sum_{l=1}^j s_{1,l} + s_{1,\overline{n}} + \sum_{l=i}^{n-1} s_{i,\overline{n}} \leq m_1 + \cdots + m_{n-2} + m_n, ~ 1 \leq j \leq n-2,~ j < i \leq n-1
\end{gather*}
and the degree inequalities
\[
\sum_{l=1}^r s_{1,l} + s_{1,\overline{n}} + \sum_{l=i}^{n-1} s_{i,\overline{n}} \leq m_1 + \cdots + m_{n-2} + m_n + \sum_{l = i}^r m_l, \quad 2 \leq i \leq n-2, ~ i \leq r \leq n-2.
\]
\end{rem}
 
 For every positive root $\alpha$, we fix a non-zero element $f_{\alpha} \in \n^-$. Explicitly, the elements $f_{\alpha}$ are given by
  \begin{gather}
     f_{i,j} = E_{j+1,i} - E_{n+j+1,n+i}, \quad 1 \leq i \leq j \leq n-1, \\
     f_{i,\overline{j}} = E_{n+i,j} - E_{n+j,i}, \quad 1 \leq i < j \leq n.
 \end{gather}
  Observe that we also have the operator $ f_{i,\overline{n}} f_{i,n-1} = f_{i,n-2} f_{i,\overline{n-1}}  =E_{n+i,i} $ where the multiplication is defined in the universal enveloping algebra. We define a total order on the generators of $S(\n^-)$ by:
  
  \begin{gather*}
      f_{n-1,\overline{n}} \succ f_{n-1,n-1} \succ \\
      f_{n-2, \overline{n-1}} \succ f_{n-2, \overline{n}} \succ f_{n-2,n-1} \succ f_{n-2,n-2} \succ \\
      f_{n-3,\overline{n-2}} \succ f_{n-3, \overline{n-1}} \succ f_{n-3, \overline{n}} \succ f_{n-3,n-1} \succ f_{n-3,n-2} \succ f_{n-3,n-3} \succ \\
      \cdots \succ \cdots \succ \cdots \succ \\
      f_{1,\overline{2}} \succ f_{1,\overline{3}} \succ \cdots \succ f_{1,\overline{n}} \succ f_{1,n-1} \succ \cdots \succ f_{1,2} \succ f_{1,1}. 
  \end{gather*}
By restriction, we obtain a total order on the generators of $S(\n_w^-)$. We use the same symbol $\succ$ for the induced homogeneous lexicographical order.  

\subsection{Polytopes} We have the following version of paths for type $D$.

\begin{defn}\label{typedpath}
 A \emph{Dyck path} is a sequence of roots of the form
 \begin{gather*}
     (\alpha_{1}, \ldots, \alpha_{1,j}), \quad j=1 , \ldots, n-1; \quad (\alpha_{1}, \alpha_{1,2}, \ldots, \alpha_{1,n-2}, \alpha_{1,\overline{n}}); \\
     (\alpha_{1}, \alpha_{1,2}, \ldots, \alpha_{1,n-1}, \alpha_{1,\overline{n-1}}); \quad (\alpha_{1}, \alpha_{1,2}, \ldots, \alpha_{1,n-2}, \alpha_{1,\overline{n}}, \alpha_{1,\overline{n-1}}).
 \end{gather*}
 A \emph{degree path} is a sequence of roots of the form
 \begin{gather*}
     (\alpha_{1}, \alpha_{1,2}, \ldots, \alpha_{1,n-2},\alpha_{1,n-1}, \alpha_{1,\overline{n-1}} , \ldots , \alpha_{1,\overline{j}}), \quad j=2, \ldots, n-2; \\
     (\alpha_{1}, \alpha_{1,2}, \ldots, \alpha_{1,n-2}, \alpha_{1,\overline{n}}, \alpha_{1,\overline{n-1}} , \ldots , \alpha_{1,\overline{j}}), \quad j=2, \ldots, n-2.
 \end{gather*}
\end{defn}

Let $\lambda = \sum_{i=1}^n m_i \m_i$ be an $\s \mathfrak{o} (2n)$ dominant integral weight. To every Dyck path we associate the following inequalities respectively:
\begin{gather}\label{ineqd1}
    \sum_{1 \leq j \leq n-1}s_{1,j} \leq m_1 + \cdots + m_j; \quad
    \sum_{j=1}^{n-2}s_{1,j} + s_{1,\overline{n}} \leq m_1 + \cdots + m_{n-2} + m_n; \\
     \sum_{j=1}^{n-1}s_{1,j} + s_{1,\overline{n-1}} \leq m_1 + \cdots + m_n; \quad
      \sum_{j=1}^{n-2}s_{1,j} + s_{1,\overline{n}} + s_{1,\overline{n-1}} \leq m_1 + \cdots  + m_n
\end{gather}
and to every degree path, we associate the following inequalities respectively:
\begin{gather}\label{ineqd2}
    \sum_{l=1}^{n-1}s_{1,l} + \sum_{l=j}^{n-1} s_{1,\overline{j}} \leq m_1 + \cdots + m_{j-1} + 2m_j + \cdots + 2m_{n-2} + m_{n-1} + m_n; \\
     \sum_{l=1}^{n-2}s_{1,l} + \sum_{l=j}^{n} s_{1,\overline{l}} \leq m_1 + \cdots + m_{j-1} + 2m_j + \cdots + 2m_{n-2} + m_{n-1} + m_n
\end{gather}
where $j = 2, \ldots, n-2$. Similar to the construction in the previous two sections, we need one more set of inequalities to define the polytopes $P_w(\lambda
)$, namely, the inequalities with coefficients. 

\begin{defn}
 Let $\textbf{p} = (\alpha_{1}, \ldots, \alpha_{1,n-1}, \alpha_{1,\bar{j}}, \ldots, \alpha_{1,\bar{n}}), ~j = 2, \ldots,n$, be a be a sequence of roots. We associate to $\textbf{p}$ tuples $t_{\textbf{p}}^k$ labelled by the path $\textbf{p}$ and a number $k \in \{ j-1, \ldots, n-2 \}$. For every number $k$, the tuple $t_{\textbf{p}}^k$ is defined by 
 \begin{itemize}
     \item $t_{\alpha_{1}} = \cdots = t_{\alpha_{1,k}} = 2$ and $t_{\alpha} = 1$ otherwise.
 \end{itemize}
 The elements of $t_{\textbf{p}}^k$ are called coefficients of $\textbf{p}$. The assignment of the coefficients is done component wise and the resulting path is called a path with coefficients.
\end{defn}

  To every path with coefficients as defined above, we associate an inequality
\begin{equation}\label{ineqd3}
   2 \sum_{l=1}^{k} s_{1,l} + \sum_{l=k+1}^{n-1} s_{1,l} + \sum_{l=j}^{n} s_{1,\overline{l}} \leq q_{\textbf{p}}^{\lambda} + \sum_{l=1}^k m_l
\end{equation}
where 
\[
q_{\textbf{p}}^{\lambda} = m_1 + \cdots + m_{j-1} + 2m_j + \cdots + 2m_{n-2} + m_{n-1} + m_n.
\]

 \begin{defn}
  We define $P_w(\lambda)$ to be the polytope given by the inequalities \ref{ineqd1} to \ref{ineqd3}.
 \end{defn}
 
As in the previous sections, we denote by $S_w(\lambda)$ and $R_i$ the sets
 \[
S_w(\lambda) = P_w(\lambda) \cap \mathbb{Z}_{\geq0}^{\#R_w^-}, \quad R_i = \{ \alpha \in R_w^- : (\m_i , \alpha) \neq 0 \}.
 \]
 For every $\alpha \in R_i$, we denote by $\textbf{m}_{\alpha}$ the multi-exponent $(m_{\alpha})_{\alpha \in R_w^-} \in \mathbb{Z}_{\geq0}^{\#R_w^-}$ defined by $m_{\alpha} = 1$ and zero otherwise. The following lemma follows immediately from the definition of $P_w(\m_i)$.
 
 \begin{lem}
The following points are in $S_w(\m_i)$.
 \begin{enumerate}
 \item For $i = 1$, the points $\textbf{m}_{\alpha}$ for all $\alpha \in R_i$ and the point $\textbf{m}_{\alpha_{1,n-1}} + \textbf{m}_{\alpha_{1,\bar{n}}}$.
     \item For $ 2 \leq i \leq n-2 $, the points 
            \begin{itemize}
                \item $\textbf{m}_{\alpha}, ~\alpha \in R_i$ and the point $\textbf{m}_{\alpha_{1,n-1}} + \textbf{m}_{\alpha_{1,\bar{n}}}$,
                \item the points $\textbf{m}_{\alpha_{1,j}} + \textbf{m}_{\alpha_{1,\bar{j'}}},~ i \leq j \leq n-1,~ 2 \leq j' \leq i$,
                \item the points $\textbf{m}_{\alpha_{1,\bar{j}}} + \textbf{m}_{\alpha_{1,\bar{j'}}},~ j \neq j',~ 2 \leq j, j' \leq i$
            \end{itemize}
     \item For $i = n-1$, the point $\textbf{m}_{\alpha_{1,n-1}}$ and the points $\textbf{m}_{\alpha_{1,\bar{j}}},~ 2 \leq j \leq n-1$.
     \item For $i=n$, the points $\textbf{m}_{\alpha_{1,\bar{j}}},~ 2 \leq j \leq n$.
 \end{enumerate}
 \end{lem}

\subsection{Differential operators}\label{differentialoperators} In the following, we describe some differential operators similar to the ones defined in \cite{FFLa1} for type $A$ and \cite{FFLa2} for type $C$. Similarly as in the $A_n$ case and $C_n$ case, the natural action of $U(\n^+)$ on the highest weight module $V(\lambda)$ induces the structure of a $U(\n^+)$-module on the associated graded space $\operatorname{gr}V(\lambda)$. With the identification
\[
S(\n^-) \simeq S(\g) / S(\n^-)S_+ (\mathfrak{h} \oplus \n^+)
\]
where $S_+ (\mathfrak{h} \oplus \n^+) \subset S (\mathfrak{h} \oplus \n^+)$ is the maximal homogeneous ideal of polynomials without constant term, we obtain an action of $U(\n^+)$ on $S(\n^-)$. We define differential operators $\partial_{\beta}$ by
\[
\partial_{\beta} f_{\alpha} = \begin{cases}
f_{\alpha - \beta}, \quad \text{if} \quad \alpha - \beta \in R^+ \\
0, \quad \text{otherwise}.
\end{cases}
\]
The operators $\partial_{\beta}$ satisfy the property $\partial_{\beta} f_{\alpha} = c_{\beta, \alpha} (\operatorname{ad}e_{\beta})(f_{\alpha})$ where $c_{\beta, \alpha}$ are some non-zero constants. One calculates,

\begin{lem}
The only non-zero vectors of the form $\partial_{\beta} f_{\alpha}$, $\beta , \alpha \in R^+$ are the following vectors: for $\alpha = \alpha_{i,j}, 1 \leq i \leq j \leq n-1$,
\begin{equation}
    \partial_{s,j} f_{i,j} = f_{i,s-1}, i < s \leq j, \quad \partial_{i,s} f_{i,j} = f_{s+1,j}, i \leq s < j,
\end{equation}
and for $\alpha = \alpha_{i,\overline{j}}, 1 \leq i < j \leq n$,
\begin{gather}
    \partial_{j,s} f_{i,\overline{j}} = f_{i,\overline{s+1}}, i < j \leq s, \quad \partial_{i,s} f_{i,\overline{j}} = f_{s+1,\overline{j}}, s < j-1, \quad \partial_{i,s} f_{i,\overline{j}} = f_{j,\overline{s+1}}, j \leq s , \\
     \partial_{j,\overline{s}} f_{i,\overline{j}} = f_{i,s-1}, i < j < s, \quad \partial_{i,\overline{s}} f_{i,\overline{j}} = f_{j,s-1}, j<s, \quad \partial_{i,\overline{s}} f_{i,\overline{j}} = f_{i,s-1}, j < s .
\end{gather}
\end{lem}
Note that although using the same analysis, the elements $\partial_{\beta} f_{\alpha}$ do not coincide with those defined in \cite{FFLa1} and \cite{FFLa2}. 
 
 \section{Minkowski Sums}\label{section6}
 Recall the notation $S_w(\lambda),~R_i$ for the sets
 \[
S_w(\lambda) = P_w(\lambda) \cap \mathbb{Z}_{\geq0}^{\#R_w^-}, \quad R_i = \{ \alpha \in R_w^-: (\m_i , \alpha) \neq 0 \}.
\]
Let $\lambda = \sum_{j=1}^n m_j \m_j$ be a dominant integral weight of the Lie algebra $\g$ of type $A, ~B,~C$ or $D$ and let $i$ be minimal such that $m_i > 0$. For a point $\textbf{s} \in S_w(\lambda)$, let $R_i^{\textbf{s}}$ denote the set
\[
R_i^{\textbf{s}} = \{ \alpha \in R_i : \textbf{s}_{\alpha} \neq 0 \}.
\]

Define a partial order $\ll$ on $R^+$ by
\[
\alpha_{i_1, j_1} \ll \alpha_{i_2, j_2} \Leftrightarrow i_1 \leq i_2 \wedge j_1 \leq j_2.
\]
Note that the partial order $\ll$ becomes a total order when restricted to the subset $R_w^-$ for all types $A, B, C$ and $D$. Recall the Minkowski sum of two sets of points $A$ and $B$ in $\mathbb{R}^n$ is defined by
\[
A + B := \{ x + y : x \in A, y \in B \}.
\]
The aim of this section is to prove that our polytopes satisfy the Minkowski property. The main ingredient is the following procedure of writing an element $\textbf{s} \in S(\lambda)$ as a sum of elements $\textbf{s} = \sum_{i=1}^n \textbf{m}_i$ such that $\textbf{m}_i \in S_w(\m_i)$.

\begin{prop}\label{minkowskidecomposition}
 Let $\lambda = \sum_{l = i}^n m_i \m_i$ be a dominant integral $\g \operatorname{-weight}$ such that $i$ is minimal with $m_i > 0$.  Let $\textbf{s} \in S_w(\lambda),~ \textbf{m}_{i} \in S_w(\m_i)$, then $\textbf{s} - \textbf{m}_{i} \in S_w(\lambda - \m_i)$.
 \end{prop}
 
\begin{proof}
  Let $\textbf{s}'$ be such that $\textbf{s} = \textbf{s}' + \textbf{m}_i$ and let $\textbf{p}$ be a path. We have to show that $\textbf{s}'$ satisfies the defining inequality for $S_w(\lambda - \m_i)$ given by $\textbf{p}$.  For $\alpha_{i,j} \in R_i$, we sometimes denote $\textbf{m}_{\alpha} := \textbf{m}_{i,j}$, etc.
  
\begin{itemize}
 \item Type $A$: If $\textbf{p}$ is a Dyck path, the proof works in the same way as in \cite{FFLa1} ( see proof of Proposition 2), hence we may assume $\textbf{p}$ is an degree path. The upper bound for the defining inequality (\ref{ineq-2}) for $S_w(\lambda)$ associated to $\textbf{p}$ is $q_{\textbf{p}}^{\lambda} = \sum_{j=i}^n m_j + \sum_{j=r}^s m_j$ where for $ 2 \leq r \leq n-1 ,\alpha_{1,s} , \alpha_{r,n} \in \textbf{p} $ for $  r \leq s \leq n-1$. The restriction of the partial order $\ll$ to the set $R_w^-$ is a total order,which implies $R_i^{\textbf{s}}$ has a unique minimal element. Let $\alpha_{1,p} \in R_i^{\textbf{s}}$ be the unique minimal element and suppose there are some $\alpha \in R_i^{\textbf{s}}$ such that $\textbf{m}_{1,p} + \textbf{m}_{\alpha} \in S_w(\m_i)$. Let $\alpha_{q,n}$ be minimal among all such roots $\alpha$. Then we have $q_{\textbf{p}}^{\lambda - \m_i} = \sum_{j=i}^n m_j -1 + \sum_{j=i}^s m_j -1$ and 
\[
\sum_{\alpha \in \textbf{p}} s_{\alpha}' = \sum_{\alpha \in \textbf{p}} s_{\alpha} - m_{1,p} - m_{q,n} \leq q_{\textbf{p}}^{\lambda} - 2 = q_{\textbf{p}}^{\lambda - \m_i}.
\]
Therefore, $\textbf{s}'$ satisfies the defining inequality for $S_w(\lambda - \m_i)$ given by $\textbf{p}$. Now suppose there are no such roots $\alpha \in R_{i}^{\textbf{s}}$ such that $\textbf{m}_{1,p} + \textbf{m}_{\alpha} \in S_w(\m_i)$, then $q_{\textbf{p}}^{\lambda - \m_i} = \sum_{j=i}^n m_j -1 + \sum_{j=r}^s m_j$ and
\[
\sum_{\alpha \in \textbf{p}} s_{\alpha}' = \sum_{\alpha \in \textbf{p}} s_{\alpha} - m_{1,p} \leq q_{\textbf{p}}^{\lambda} - 1 = q_{\textbf{p}}^{\lambda - \m_i}
\]
therefore, $\textbf{s}'$ satisfies the defining inequality for $S_w(\lambda - \m_i)$ given to $\textbf{p}$. 

\item Type $B$ and $C$: The description of the sets $S_w(\m_i)$ is given in Lemma \ref{fundamentalpointsB} for type $B$ and Lemma \ref{fundamentalpoints} for type $C$. Let $\alpha_{1,p} \in R_i^{\textbf{s}}$ be minimal with $s_{1,p} \neq 0$. If $\textbf{p}$ is a Dyck path, we take $\textbf{m}_i = \textbf{m}_{1,p}$ and it follows using similar calculations as above for type $A$ that $\textbf{s}'$ satisfies the defining inequality for $S_w(\lambda - \m_i)$ given by $\textbf{p}$. Now suppose $\textbf{p}$ is a degree path or a path with coefficients. Assume further that there are some $\alpha \in R_i^{\textbf{s}}$ such that $\textbf{m}_{1,p} + \textbf{m}_{\alpha} \in S_w(\m_i)$. Let $\alpha_{1,\overline{q}}$ be minimal among all such $\alpha$. Then we take $\textbf{m}_i = \textbf{m}_{1,p} + \textbf{m}_{1,\overline{q}}$ and it suffices to note that by definition of the inequalities defining the polytope $P_w(\lambda)$ (see Section \ref{section3} for type $C$ and Section \ref{section4} for type $B$), there are enough copies of $m_i$ and it follows that the point $\textbf{s}'$ satisfies the required inequalities for $S_w(\lambda - \m_i)$ given by $\textbf{p}$.  Suppose there are no such $\alpha$ such that $\textbf{m}_{1,p} + \textbf{m}_{\alpha} \in S_w(\m_i)$. Let $\tilde{\textbf{p}}$ be the path
 \begin{gather*}
    \tilde{\textbf{p}} = (\alpha_1, \ldots , \alpha_{1,i+1}, \alpha_{1,\overline{1}}) \quad \text{for type C} \\
    \tilde{\textbf{p}} = (\alpha_1, \ldots , \alpha_{1,i+1}) \quad \text{for type B}.
 \end{gather*}
We have that $\textbf{s}$ is supported on $\tilde{\textbf{p}}$, i.e. $s_{\alpha} = 0$ if $\alpha \notin \tilde{\textbf{p}}$.  Then we can take $\textbf{m}_i = \textbf{m}_{1,p}$ and it follows that $\textbf{s}'$ satisfies the defining inequality for $S_w(\lambda - \m_i)$ given by $\tilde{\textbf{p}}$. 

\item Type $D$: The proof works in the same way as above for type $B$ and $C$. The only difference is that a path $\textbf{p}$ by definition (see Definition \ref{typedpath}) can contain at most one root $\alpha_{1,n-1}$ or $\alpha_{1,\overline{n}}$ and the proof follows using similar calculations as above.
\end{itemize}
\end{proof}

\begin{lem}\label{minkowski1}
$S_w(\lambda) + S_w(\mu) = S_w(\lambda + \mu)$.
\end{lem}

\begin{proof}
 Follows immediately by Proposition \ref{minkowskidecomposition}.
\end{proof}

 \section{Main Theorem}\label{section7}
 Let $q_{\textbf{p}}^{\lambda}$ denote the right hand side of the inequality given by a path $\textbf{p}$. Recall from Section \ref{section1} that $ S(\n_w^-)v_{\lambda} \simeq S(\n_w^-) / I(\lambda)$ where $I(\lambda)$ is the annihilating ideal. Our main goal in this section is to prove the following theorem. 
 
  \begin{thm}\label{basis}
 Let $\lambda = \sum_{j=1}^n m_j \m_j$ be a dominant integral $\g \operatorname{-weight}$. 
 \begin{itemize}
 \item[i)] The set $ \{ f^{\textbf{s}} v_{\lambda}: \textbf{s} \in S_w(\lambda) \}$ forms a basis of $S(\n_w^-) v_{\lambda}$.
 \item[ii)] The ideal $I(\lambda)$ is a monomial ideal generated by the monomials $ f^{\textbf{s}} $ where $\textbf{s}$ is a multi-exponent supported on a path $\textbf{p}$ and 
 \[
 \sum_{\alpha \in \textbf{p}} s_{\alpha} = q_{\textbf{p}}^{\lambda} +1.
 \]
 \end{itemize}
 \end{thm}
 
 In Subsection \ref{linearindependnce}, we prove that the set $S_w(\lambda)$ labels a linearly independent set. In Subsection \ref{spanningproperty}, we prove that $S_w(\lambda)$ labels a spanning set.
 
  \subsection{Linear independence}\label{linearindependnce}
 We will use the following notion of essential vectors, see \cite{Gor} and \cite{CF}. Recall the total order $\succ$ on the generators $f_{\alpha},~ \alpha \in R_w^-$ of $S(\n_w^-)$. We denote by the same symbol $\succ$ the induced homogeneous lexicographical order. Let $\textbf{a} \in \mathbb{Z}_{\geq 0}^m, ~ m= \# R_w^-$.
 
 \begin{defn}
  We call $f^{\textbf{a}}v_{\lambda}$ an essential vector if
  \[
  f^{\textbf{a}} v_{\lambda} \notin \operatorname{span}_{\mc} \{ f^{\textbf{a}'} v_{\lambda} : f^{\textbf{a}'} \prec f^{\textbf{a}} \}.
  \]
 \end{defn}
 
 By construction, the essential vectors $f^{\textbf{a}}v_{\lambda}$ are linearly independent. For any two dominant integral $\g \operatorname{-weights} \lambda, ~\mu$, let $U(\n_w^-)(v_{\lambda} \otimes v_{\mu}) \subset U(\n_w^-)v_{\lambda} \otimes U(\n_w^-) v_{\mu}$ be the Cartan component. The following proposition is proved in \cite{FFLb} (see Proposition 2.11).

\begin{prop}\label{ess}  If $f^{\textbf{a}}v_{\lambda}$ and $f^{\textbf{a}'}v_{\mu}$ are essential vectors, then the vector $f^{\textbf{a}+\textbf{a}'}(v_{\lambda} \otimes v_{\mu})$ is also essential.
\end{prop}

We showed in Section \ref{section6} that $S_w(\lambda + \mu) = S_w(\lambda) + S_w(\lambda)$. Then by Proposition \ref{ess}, to show that $S_w(\lambda)$ labels a linearly independent set, it is enough to show that the the fundamental sets $S_w(\m_j)$ label linearly independent sets.

\begin{prop}\label{lin-indep}
The vectors $f^{\textbf{s}}v_{\lambda}$ with $\textbf{s} \in S_w(\lambda)$ are essential.
\end{prop}

\begin{proof}
 The description of $S_w(\m_i)$ is given in Section \ref{section2}, \ref{section3}, \ref{section4} and \ref{section5} for $\g$ of Lie type $A,~ B, ~C$ and $D$ respectively. With respect to the homogeneous order $\succ$, it is clear that 
\[
  f^{\textbf{s}} v_{\m_i} \notin \operatorname{span}_{\mc} \{ f^{\textbf{s}'} v_{\m_i} : f^{\textbf{s}'} \prec f^{\textbf{s}} \},
  \]
hence the vectors $f^{\textbf{s}}v_{\m_i}$ are essential. We obtain by Proposition \ref{ess} together with the Minkowski property that the vectors $f^{\textbf{s}}v_{\lambda},~ \textbf{s} \in S_w(\lambda)$ are essential as well and therefore linearly independent.
\end{proof}

\subsection{Spanning property}\label{spanningproperty}
The goal of this subsection is to prove that the elements $f^{\textbf{s}}v_{\lambda}, ~ \textbf{s} \in S_w(\lambda)$ span $S(\n_w^-)v_{\lambda}$. We will show that $f^{\textbf{s}} v_{\lambda} = 0$ for any $\textbf{s} \notin S_w(\lambda)$. We begin with the following proposition.

\begin{prop}\label{ideal}
Let $\textbf{p}$ be a Dyck path or degree path and let $\textbf{s}$ be a multi-exponent supported on $\textbf{p}$, i.e. $s_{\alpha} = 0$ if $\alpha \notin \textbf{p}$. Assume further that
\[
\sum_{\alpha  \in \textbf{p}} s_{\alpha} = q_{\textbf{p}}^{\lambda} +1,
\]
then $f^{\textbf{s}} v_{\lambda} = 0$.
\end{prop}

\begin{proof}
    First suppose $\textbf{p}$ is a degree path. By Remark \ref{idealremmark}, $q_{\textbf{p}}^{\lambda}$ is the maximum PBW degree that can be achieved and exceeding this degree implies we obtain a trivial action. 
    
    Now suppose $\textbf{p}$ is a Dyck path. Our strategy is as follows: let $f_{\alpha}^{\lambda(h_{\alpha})+1}v_{\lambda} = 0 ~\operatorname{in}~ U(\n_w^-)v_{\lambda} \subset U(\n^-)v_{\lambda}$, so $f_{\alpha}^{\lambda(h_{\alpha})+1} \in I(\lambda)$. We use differential operators $\partial_{\alpha}$ which satisfy the property $\partial_{\alpha} f_{\beta} = c_{\alpha, \beta} \operatorname{ad}(e_{\alpha})(f_{\beta})$ for some non-zero constants $c_{\alpha , \beta}$. By definition of our Dyck paths, we can apply these differential operators to $f_{\alpha}^{\lambda(h_{\alpha})+1}$ to obtain new elements in $I(\lambda)$. We consider the four cases.
    \begin{itemize}
        \item Type $A$: for all positive roots $\alpha_{i,j} = \alpha_i + \cdots + \alpha_j \in R_w^-$, we have 
        \[
        f_{i,j}^{m_i + \cdots + m_j +1}v_{\lambda} = 0 \quad \text{therefore} \quad f_{i,j}^{m_i + \cdots + m_j +1} \in I(\lambda).
        \]
        The operators $\partial$ are described in \cite{FFLa1} (proof of Proposition 1.). If $\alpha_{1,n} \in \textbf{p}$, then we apply differential operators to $f_{1,n}^{m_1 + \cdots + m_n +1}$, for instance if $\textbf{p} = (\alpha_1 , \cdots , \alpha_{1,n}, \alpha_n)$, we have
        \[
        \partial_{1,n-1}^{s_{n,n}} \partial_{n,n}^{s_{1,n-1}} \ldots \partial_{3,n}^{s_{1,2}} \partial_{2,n}^{s_{1,1}} f_{1,n}^{m_1 + \ldots + m_n  +1} \in I(\lambda)
        \]
        is proportional with a non-zero constant to 
        \[
        f_{1,1}^{s_{1,1}} f_{1,2}^{s_{1,2}} \ldots f_{1,n}^{s_{1,n}} f_{n,n}^{s_{n,n}} \in I(\lambda).
        \]
        For the remaining Dyck paths, consider the paths
        \[
        \textbf{p}_1 = (\alpha_1, \ldots , \alpha_{1,j}),~ 2 \leq j \leq n-1, \quad \text{and} \quad \textbf{p}_2 = (\alpha_{i,n}, \ldots , \alpha_n),~ 2 \leq i \leq n-1.
        \]
        Then we apply differential operators on $f_{1,j}^{m_1 + \cdots + m_j +1}$ and $f_{i,n}^{m_i + \cdots + m_n +1}$ respectively and obtain 
        \[
        \partial_{j,j}^{s_{1,j-1}}  \ldots \partial_{3,j}^{s_{1,2}} \partial_{2,j}^{s_{1,1}} f_{1,j}^{m_1 + \ldots + m_{j}  +1} \quad \text{and} \quad 
        \partial_{i,i}^{s_{i+1,n}}  \ldots \partial_{i,n-2}^{s_{n-1,n}} \partial_{i,n-1}^{s_{n,n}} f_{i,n}^{m_i + \ldots + m_{n}  +1}
        \]
        which is proportional with a non-zero constant to
        \[
        f_{1,1}^{s_{1,1}} f_{1,2}^{s_{1,2}} \ldots f_{1,j}^{s_{1,j}} , ~
        f_{i,n}^{s_{i,n}} f_{i+1,n}^{s_{i+1,n}} \ldots f_{n-1,n}^{s_{n-1,n}} f_{n,n}^{s_{n,n}} \in I(\lambda)
        \]
        respectively.
        
        \item Type $C$: The differentials $\partial$ have been considered in \cite{FFLa2} and the action is given explicitly in Lemma 2.2. of the same paper. If all roots appearing in the path are also roots for the subroot system of type $A$, then $\textbf{p}$ is also a type $A$ Dyck path and the result follows from the corresponding result for type $A$ considered above. For the Dyck path $(\alpha_1, \ldots , \alpha_{1,n-1}, \alpha_{1,\overline{1}})$, since
        \[
        f_{1,\overline{1}}^{m_1 + \cdots + m_n + 1} v_{\lambda} = 0 
        \]
        we obtain the expression 
        \[
        \partial_{1,\overline{2}}^{s_{1,1}} \cdots \partial_{1,\overline{n-1}}^{s_{1,n-2}} \partial_{1,\overline{n}}^{s_{1,n-1}} f_{1,\overline{1}}^{m_1 + \cdots + m_n + 1} \in I(\lambda)
         \]
        which is proportional with a non-zero constant to
        \[
         f_{1}^{s_{1,1}} \cdots f_{1,n-2}^{s_{1,n-2}} f_{1,n-1}^{s_{1,n-1}} f_{1,\overline{1}}^{s_{1,\overline{1}}}. 
        \]
        
        \item Type $D$: The operators $\partial$ are discussed in Subsection \ref{differentialoperators}. Similarly as the type $C$ case, if $\textbf{p}$ is a path in the subroot system of type $A$, then the result follows from the corresponding result in type $A$. Consider the Dyck paths $(\alpha_1, \ldots , \alpha_{1,n-2}, \alpha_{1,\overline{n}}), (\alpha_1, \ldots , \alpha_{1,n-1}, \alpha_{1,\overline{n-1}})$ and $(\alpha_1, \ldots , \alpha_{1,n-2}, \alpha_{1,\overline{n}} \alpha_{1,\overline{n-1}})$, since 
        \[
        f_{1,\overline{n}}^{m_1 + \cdots + m_{n-2} + m_n +1}, \quad  f_{1,\overline{n-1}}^{m_1 + \cdots + m_n +1} \in I(\lambda),
        \]
        we obtain the expressions
        \begin{gather*}
        \partial_{1,\overline{2}}^{s_{1,1}} \partial_{1,\overline{3}}^{s_{1,2}} \cdots \partial_{1,\overline{n-2}}^{s_{1,n-3}} \partial_{1,\overline{n-1}}^{s_{1,n-2}} f_{1,\overline{n}}^{m_1 + \cdots + m_{n-2} + m_n + 1}, \quad \partial_{1,\overline{2}}^{s_{1,1}} \partial_{1,\overline{3}}^{s_{1,2}} \cdots \partial_{1,\overline{n-1}}^{s_{1,n-2}} \partial_{1,\overline{n}}^{s_{1,n-1}} f_{1,\overline{n-1}}^{m_1 + \cdots + m_n + 1} , \\
        \partial_{1,\overline{2}}^{s_{1,1}} \partial_{1,\overline{3}}^{s_{1,2}} \cdots \partial_{1,\overline{n-2}}^{s_{1,n-3}} \partial_{1,\overline{n-1}}^{s_{1,n-2}}
        \partial_{n-1,n-1}^{s_{1,\overline{n}}}
        f_{1,\overline{n-1}}^{m_1 + \cdots + m_n + 1} ,
        \end{gather*}
        in $I(\lambda)$ are proportional with a non-zero constant to
        \begin{gather*}
        f_{1,1}^{s_{1,1}} f_{1,2}^{s_{1,2}} \cdots f_{1,n-3}^{s_{1,n-3}} f_{1,n-2}^{s_{1,n-2}} f_{1,\overline{n}}^{s_{1,\overline{n}}} , \quad
        f_{1,1}^{s_{1,1}} f_{1,2}^{s_{1,2}} \cdots f_{1,n-2}^{s_{1,n-2}} f_{1,n-1}^{s_{1,n-1}} f_{1,\overline{n-1}}^{s_{1,\overline{n-1}}} , \\
        f_{1,1}^{s_{1,1}} f_{1,2}^{s_{1,2}} \cdots f_{1,n-3}^{s_{1,n-3}} f_{1,n-2}^{s_{1,n-2}}
        f_{1,\overline{n}}^{s_{1,\overline{n}}}
        f_{1,\overline{n-1}}^{s_{1,\overline{n-1}}} 
        \end{gather*}   
        respectively.
        
        \item Type $B$: For type $B$, all Dyck paths are also type $A$ Dyck paths, hence the result follows using the corresponding result for type $A$.
    \end{itemize}
\end{proof}

Denote by $P_w'(\lambda)$ and $S_w'(\lambda)$ the polytope defined by the inequalities arising from Dyck paths and degree paths except those with coefficients (i.e an inequality will not appear if it has coefficients greater than one) and the set of lattice points respectively.We still obtain the Minkowski property for this polytope using similar considerations used in Proposition \ref{minkowskidecomposition}.

\begin{prop}\label{polytopeminus}
Let $\textbf{s} \in S_w'(\lambda) \setminus S_w(\lambda)$ and let
\[
\textbf{s} = \textbf{s}_1 + \cdots + \textbf{s}_n
\]
be a Minkowski sum decomposition of $\textbf{s}$. There exists $\textbf{s}_d$ in this decomposition such that $\textbf{s}_d \in S_w'(\m_d) \setminus S_w(\m_d)$ and $f^{\textbf{s}_d} v_{\m_d} = 0$.
\end{prop}

\begin{proof}
Suppose all the $\textbf{s}_d$ in the decomposition belong to $S_w(\m_d)$. This implies $\textbf{s} \in S_w(\lambda)$, a contradiction since $\textbf{s} \in S_w'(\lambda) \setminus S_w(\lambda)$. A straightforward computation in each type shows that $f^{\textbf{s}_d} v_{\m_d} =0$.
\end{proof}

\begin{lem}\label{trivialactionc}
For every $\textbf{s} \in S_w'(\lambda) \setminus S_w(\lambda)$, $f^{\textbf{s}} v_{\lambda} = 0$.
\end{lem}

\begin{proof}
 By Proposition \ref{polytopeminus}, we can write $\textbf{s}$ as a Minkowski sum decomposition such that one of the summands $\textbf{s}_d \in S_w'(\m_d) \setminus S_w(\m_d)$ for some $d$. Let us consider the expansion of the expression $f^{\textbf{s}'} f^{\textbf{s}_d} (v_{\m_d} \otimes v_{\lambda - \m_d})$ where $ \textbf{s} = \textbf{s}' + \textbf{s}_d, \textbf{s}_d \in S_w'(\m_d) \setminus S_w(\m_d) $.  That is the expression
 \begin{equation}\label{expansion}
     ( f^{\textbf{s}'} f^{\textbf{s}_d} v_{\m_d}) \otimes v_{\lambda - \m_d} + v_{\m_d} \otimes (f^{\textbf{s}'} f^{\textbf{s}_d}  v_{\lambda - \m_d}) + 
    f^{\textbf{s}'} v_{\m_d} \otimes  f^{\textbf{s}_d}  v_{\lambda - \m_d} +
    f^{\textbf{s}_d} v_{\m_d} \otimes  f^{\textbf{s}'}  v_{\lambda - \m_d} + \text{ rest}.
 \end{equation}
 It is clear that the first four summands are all equal to zero. We have to show that the remaining terms rest in (\ref{expansion}) are also zero. The terms appearing in rest are terms of the form
 \begin{equation}
     f^{\textbf{s}_d'} v_{\m_d} \otimes  f^{\textbf{s}''}  v_{\lambda - \m_d} , \quad \textbf{s}_d' + \textbf{s}'' = \textbf{s},
 \end{equation}
 where some of the coordinates of $\textbf{s}_d'$ and $\textbf{s}''$ differ from those of $\textbf{s}_d$ and $\textbf{s}'$ respectively. Since we have a homogeneous order $\succ$. We have that either 
 \[
 \textbf{s}_d \succ \textbf{s}_d' \quad \text{or} \quad \textbf{s}' \succ \textbf{s}''.
 \]
 We prove the case $\textbf{s}' \succ \textbf{s}''$. The other case is completely similar. There are two possibilities:
 \begin{itemize}
     \item[i)] $\sum \textbf{s}' > \sum \textbf{s}''$. This implies $\sum \textbf{s}_d' > \sum \textbf{s}_d$ in which case this term is zero. 
     \item[ii)] $\sum \textbf{s}' = \sum \textbf{s}''$. This implies that $\sum \textbf{s}_d' = \sum \textbf{s}_d$. Then either $\textbf{s}_d' \notin S_w(\m_d)$ or $\textbf{s}_d' \in S_w(\m_d)$. If $\textbf{s}_d' \notin S_w(\m_d)$, then this term is zero hence we may assume $\textbf{s}_d' \in S_w(\m_d)$. Similarly, if $\textbf{s}'' \notin S_w'(\lambda - \m_d)$, then $\textbf{s}''$ violates at least a Dyck path inequality or a degree path inequality. By Proposition \ref{ideal}, this implies this term is zero hence we may assume $\textbf{s}'' \in S_w'(\lambda - \m_d)$. Now apply Proposition \ref{polytopeminus} to $\textbf{s}''$ and repeat the procedure above. After finitely many steps, we arrive at a point $\textbf{t} $ such that either $\textbf{t} \in S_w'(\m_{d'}) \setminus S_w(\m_{d'})$ or $\textbf{t} \notin S_w' (\m_{d'})$.
 \end{itemize}
\end{proof}
 
 \begin{cor}
    The elements $f^{\textbf{s}}v_{\lambda},~ \textbf{s} \in S_w(\lambda)$ span the module $S(\n_w^-)v_{\lambda}$.
 \end{cor}
 
 \begin{proof}
We know that $S(\n_w^-)v_{\lambda}$ is spanned by elements of the form $f^{\textbf{s}}v_{\lambda}$ where $\textbf{s}$ is an arbitrary multi-exponent. Let $\textbf{s} \notin S_w(\lambda)$, then there exists a path $\textbf{p}$ such that $\sum_{\alpha_{i,j} \in \textbf{p}} s_{i,j} > q_{\textbf{p}}^{\lambda}$. By Proposition \ref{ideal} and Lemma \ref{trivialactionc}, $f^{\textbf{s}} v_{\lambda} = 0$. It follows that the points $\textbf{s} \in S_w(\lambda)$ span the module $S(\n_w^-)v_{\lambda}$ and the ideal $I(\lambda)$ is a monomial ideal generated by the monomials $f^{\textbf{s}}$ such that
\[
\sum_{\alpha \in \textbf{p}} s_{\alpha} = q_{\textbf{p}}^{\lambda} + 1.
\]
\end{proof}

\begin{prop} \label{apoly}
For any $r_{\gamma} = s_{\alpha_i} s_{\alpha_{i+1}} \cdots s_{\alpha_n} \cdots s_{\alpha_{i+1}} s_{\alpha_i},~ 1 \leq i \leq n, ~ S(\n_{r_{\gamma}}^-) v_{\lambda}$ has a basis labelled by a normal polytope $P_{r_{\gamma}}(\lambda)$ such that the annihilating ideal of $S(\n_{r_{\gamma}}^-) v_{\lambda}$ is monomial. Furthermore, for any subword $u = s_{\alpha_k} s_{\alpha_{k+1}} \cdots s_{\alpha_n} \cdots s_{\alpha_i}, ~ k>i$, we have $P_u(\lambda)$ is a face of $P_{r_{\gamma}}(\lambda)$ via the embedding $P_u(\lambda) \hra P_{r_{\gamma}}(\lambda)$ defined by
\[
\textbf{s} = (s_{\alpha})_{\alpha \in R_u^-} \mapsto \textbf{s}' = (t_{\alpha})_{\alpha \in R_{r_{\gamma}}^-}  \quad \text{where} \quad t_{\alpha} := 
\begin{cases}
s_{\alpha} & , \text{if } \alpha \in R_u^- \\
0 & , \text{otherwise.}
\end{cases}
\]
\end{prop}

\begin{proof}
 The construction of $P_{r_{\gamma}}(\lambda)$ is exactly the same as for $P_w(\lambda)$ and the proofs done earlier can be repeated up to relabelling of indices. 
\end{proof}
 
 We end this section with the following application to Demazure modules.
 
 \begin{thm}
   Let $\lambda = \sum_{j=1}^n m_j \m_j$ be a dominant integral highest weight of a simple Lie algebra of type $A,B,C$ or $D$ and let $r_{\gamma} = s_{\alpha_i} s_{\alpha_{i+1}} \cdots s_{\alpha_n} \cdots s_{\alpha_{i+1}} s_{\alpha_i},~ 1 \leq i \leq n$. Then the set
   \[
   \left\{ e^{\textbf{s}} v_{{r_{\gamma}}(\lambda)} = \prod_{\alpha \in R_{r_{\gamma}}^-} ( e_{{r_{\gamma}}(\alpha)} )^{s_{\alpha}} v_{{r_{\gamma}}(\lambda)} : \textbf{s} \in S_{r_{\gamma}}(\lambda) \right\} 
   \]
   is a basis of $\operatorname{gr} V_{r_{\gamma}}(\lambda)$ and by choosing an order in each factor also of $V_{r_{\gamma}}(\lambda)$.
 \end{thm}
 
 \section{Acknowledgements} The author would like to thank Ghislain Fourier and Johannes Flake for many useful discussions and constant support. The author is supported by a scholarship from the German Academic Exchange Service (DAAD) under the grant 'Research Grants - Doctoral Programmes in Germany, 2018/19'. 
 
 \printbibliography
 

\end{document}